\documentclass[12pt,french,english]{article}
\usepackage[T1]{fontenc}
\usepackage[latin9]{inputenc}
\usepackage{babel}
\usepackage{mathrsfs}
\usepackage{url}
\usepackage{enumitem}
\usepackage{amsmath}
\usepackage{amsthm}
\usepackage{amssymb}
\usepackage{stmaryrd}
\usepackage{geometry}
\usepackage{fancyhdr}
\usepackage[bookmarks=true,bookmarksnumbered=true,bookmarksopen=true,bookmarksopenlevel=4,
 breaklinks=true,pdfborder={0 0 1},backref=section,colorlinks=false]
 {hyperref}

\makeatletter
\numberwithin{equation}{section}
\numberwithin{figure}{section}

\@ifundefined{date}{}{\date{}}
\usepackage{babel}

\newcommand{\R}{\mathbb{R}}
\newcommand{\p}{\ensuremath{\partial}}
\usepackage{graphics}
\usepackage{graphicx}\usepackage{pstricks}
\usepackage{pst-eucl}\usepackage{pst-plot}\usepackage{pst-tree}\usepackage{pstricks-add}\usepackage{tikz}

\theoremstyle{plain}
\newtheorem{thm}{\protect\theoremname}[section]
\newtheorem{prop}[thm]{\protect\propositionname}\theoremstyle{remark}
\newtheorem{rem}[thm]{\protect\remarkname}\newtheorem{notation}[thm]{\protect\notationname}\theoremstyle{plain}
\newtheorem{cor}[thm]{\protect\corollaryname}\providecommand{\corollaryname}{Corollary}
\providecommand{\notationname}{Notation}
\providecommand{\propositionname}{Proposition}
\providecommand{\remarkname}{Remark}
\providecommand{\theoremname}{Theorem}

\makeatother

\theoremstyle{definition}
\newtheorem{defn}[thm]{\protect\definitionname}
\theoremstyle{plain}
\newtheorem{lem}[thm]{\protect\lemmaname}
\addto\captionsenglish{\renewcommand{\corollaryname}{Corollary}}
\addto\captionsenglish{\renewcommand{\definitionname}{Definition}}
\addto\captionsenglish{\renewcommand{\lemmaname}{Lemma}}
\addto\captionsenglish{\renewcommand{\notationname}{Notation}}
\addto\captionsenglish{\renewcommand{\propositionname}{Proposition}}
\addto\captionsenglish{\renewcommand{\remarkname}{Remark}}
\addto\captionsenglish{\renewcommand{\theoremname}{Theorem}}
\addto\captionsfrench{\renewcommand{\corollaryname}{Corollaire}}
\addto\captionsfrench{\renewcommand{\definitionname}{Définition}}
\addto\captionsfrench{\renewcommand{\lemmaname}{Lemme}}
\addto\captionsfrench{\renewcommand{\notationname}{Notation}}
\addto\captionsfrench{\renewcommand{\propositionname}{Proposition}}
\addto\captionsfrench{\renewcommand{\remarkname}{Remarque}}
\addto\captionsfrench{\renewcommand{\theoremname}{Théorème}}
\providecommand{\corollaryname}{Corollary}
\providecommand{\definitionname}{Definition}
\providecommand{\lemmaname}{Lemma}
\providecommand{\notationname}{Notation}
\providecommand{\propositionname}{Proposition}
\providecommand{\remarkname}{Remark}
\providecommand{\theoremname}{Theorem}

\begin{document}
\title{Global-in-time existence and uniqueness of piecewise $C^{1}$ solutions
to the unsteady initial-boundary value problem for the four-velocity
planar Broadwell model in a rectangular domain for arbitrarily large data}
\author{Koudzo Togbévi Selom SOBAH\textsuperscript{\textsuperscript{1{*}}}\textsuperscript{}
and Amah Séna D'ALMEIDA\textsuperscript{\textsuperscript{2}}}

\maketitle
\textsuperscript{1,2}Department of Mathematics, Faculty of Sciences
and Laboratory of Mathematics and Applications, University of Lomé,
Lomé, TOGO

{*} corresponding email: deselium@gmail.com 
\begin{abstract}
Since the pioneering work of James E. Broadwell, discrete velocity
models have played a fundamental role in approximating the Boltzmann
equation and in the analysis of non-equilibrium gas dynamics. Despite
their apparent simplicity, many fundamental analytical questions remain
open, in particular the global existence and uniqueness of classical
or piecewise classical solutions to initial-boundary value problems,
even for the widely studied four-velocity Broadwell model.

In this paper, we establish the global-in-time existence and uniqueness
of piecewise $C^{1}$ solutions to the nonstationary four-velocity
Broadwell system in a rectangular domain. The analysis of local solutions
is carried out in a class of piecewise $C^{1}$ functions possessing
except possibly on a finite number of planes, continuous first-order
partial derivatives.

Our approach is based on fixed point arguments combined with suitable
a priori estimates that provide uniform bounds on the local solution
and its first-order partial derivatives. The bounds ensure that the
solution remains controlled for all time and can be extended globally. 

These results provide a rigorous well-posedness framework for this
prototypical discrete velocity model and contribute to a deeper understanding
of the analytical properties of discrete velocity models. 
\end{abstract}
\textbf{Key words and phrases:} discrete velocity(Boltzmann) models,
discrete kinetic equations, initial-boundary value problems, existence,
uniqueness, fixed point theorems.

\textbf{2020 Mathematics Subject Classification:} 76A02, 76M28

\section{Introduction }

Discrete kinetic equations arise from a velocity discretization of
the Boltzmann equation, in which particle velocities are assumed to
take values in a finite set of prescribed vectors. Under this assumption,
the nonlinear integro-differential Boltzmann equation \cite{botlzman}
is replaced by a system of semilinear hyperbolic equations describing
the evolution of the particle number densities corresponding to these
discrete velocities. Following the pioneering works of Broadwell \cite{3,4},
who introduced physically acceptable discrete velocity models capable
of describing real gas flows, a general theoretical framework for
binary collisions was subsequently developed \cite{1} and later extended
to encompass multiple collisions of higher order. Discrete kinetic
theory has since evolved along two complementary directions: the mathematical
analysis of kinetic equations, including existence and uniqueness
results, and the modeling and simulation of gas flow problems.

Existence theory for discrete kinetic equations, particularly for
classical or piecewise classical solutions, is essential both for
the mathematical understanding of these models and for the reliability
of numerical methods used in physics and engineering. The favorable
mathematical structure of discrete velocity models has led to a rapid
development of their theoretical analysis. In one spatial dimension,
global existence and uniqueness results for mixed problems are well
established \cite{5}\foreignlanguage{french}{ for both classical
and weak solutions.} Existence and uniqueness for the initial-boundary
value problem were proved for the two-velocity Carleman model using
fractional step techniques \cite{2}, in a weak \foreignlanguage{french}{solution
framework, for arbitrary data}. \foreignlanguage{french}{Global existence
for the Cauchy problem was established in the framework of mild solutions
for a semilinear hyperbolic system arising in kinetic theory, without
requiring small initial data }\cite{tartar}.\foreignlanguage{french}{
}Exact\foreignlanguage{french}{ classical solutions for arbitrary
data} were constructed for the three-velocity Broadwell model \cite{7},
and similarly\foreignlanguage{french}{ for discrete-velocity models
\cite{11}.}The initial value problem with an arbitrary number of
velocities was solved in \cite{Cabanne-kawashima}\foreignlanguage{french}{
for classical solutions with arbitrary large data, and global existence
results for classical solutions to initial-boundary value problems
for the discrete Boltzmann equation with arbitrary number of velocities
were established }\cite{kawasima} with arbitrary large data. \foreignlanguage{french}{Existence
and uniqueness of classical solutions for the unsteady ten-velocity
model $C_{1}$ were established \cite{14} for }arbitrary large initial
and boundary data, \foreignlanguage{french}{while the transition from
unsteady evaporation flows to steady states in the classical solution
framework was studied \cite{compte rendu meca} for arbitrary number
of velocities. }

In contrast, results in the multidimensional case remain limited.
The first stationary result was obtained for the planar four-velocity
Broadwell model on a rectangular domain using Schaefer\textquoteright s
fixed-point theorem in \cite{Cercoi illner shinbrit} within the framework
of $C^{1}$ solutions for arbitrarily large data. Subsequently, exact
classical solutions were constructed for the stationary two-dimensional
Broadwell model \foreignlanguage{french}{\cite{dameida agosse}},
while stationary solutions within a weak solution, renormalized framework
under arbitrary data were investigated in \foreignlanguage{french}{\cite{arkeryd}}.
The existence of classical solutions to an overdetermined boundary
value problem for the four-velocity Broadwell model with an orientation
of $\theta=\pi/4$ under arbitrary data was investigated in \foreignlanguage{french}{\cite{Nicou=00003D00003D0000E9}},
where exact analytic solutions were also constructed. An extension
to a stationary problem for a 15-velocity, 3-modulus discrete model
was considered for arbitrary data in \cite{d almeida}, with exact
classical solutions. More recently, the stationary problem for the
general four-velocity Broadwell model $B_{\theta}$, $\theta\in\left[0;\pi/2\right]$,
characterized by the orientation parameter $\theta$, was studied
in \foreignlanguage{french}{\cite{defoou}}, where existence and boundedness
results for classical solutions under arbitrary data, as well as non-uniqueness,
were proved. Although numerous multidimensional stationary results
for classical solutions under arbitrary data exist, they remain restricted
to specific models.

For nonstationary multidimensional problems, several contributions
have addressed different formulations and classes of solutions. The
fractional-step method was investigated in \foreignlanguage{french}{\cite{Toscani walus}}
for the two-dimensional Broadwell Cauchy problem, yielding local weak
solutions under small-data assumptions. Global existence of classical
solutions to the multidimensional Cauchy problem under small initial-data
assumptions was established in \foreignlanguage{french}{\cite{6}}
for general models, representing a significant advance in the field.
However, the initial-boundary value problem has not been sufficiently
addressed in this context. The first novelty of the present work is
to investigate this initial-boundary value problem for the Broadwell
model. The second novelty lies in establishing results for arbitrarily
large data, thereby removing the smallness restriction.

The remainder of the paper is organized as follows. In Section \ref{koooapekkke},
we formulate two initial-boundary value problems---one on a bounded
time interval and the other on an unbounded time interval---and define
the admissible classes of data and solutions for our analysis. In
Section \ref{kooollooaooaiei}, we introduce a fixed-point problem
associated with the initial-boundary value problem on a bounded time
interval. The corresponding fixed-point operator is then utilized
to prove the uniqueness of solutions on the unbounded time interval
in Section \ref{sec:Uniqueness-theorem}. In Section \ref{sec:Positiveness-of-the},
we establish the positivity of the solution for positive data. In
Section \ref{sec:Existence-theorem-for}, we prove the local existence
of piecewise classical solutions for arbitrarily large data within
the considered class. Finally, in Section \ref{sec:Existence-theorem-for-1},
we establish the global existence of a bounded piecewise classical
solution on the unbounded time interval for arbitrarily large data.

\section{Two initial-boundary value problems }\label{koooapekkke}

\subsection{The discrete Boltzmann equations}

The discrete Boltzmann (or kinetic) equations describe the evolution
of a system of particles which velocities belong to a finite subset
$\left\{ \vec{u}_{i},\ i=1,\cdots,p\right\} $ of $\R^{d},d=1,2,3$,
$p\in\mathbb{N}^{*}.$ Suppose the particles are only subject to binary
collisions. If $N_{i}\left(t,\left(x_{\alpha}\right)^{d}_{\alpha=1}\right)$
denotes the number density of the particles with velocity $\vec{u}_{i},\ i=1,\cdots,p$
at time $t$ and position $M\left(x_{\alpha}\right)^{d}_{\alpha=1}$(
with $\left(x_{\alpha}\right)^{d}_{\alpha=1}\in\R^{d}$), the kinetic
equations are given by:

\begin{align}
\forall i=1,\cdots,p,\nonumber \\
\dfrac{\partial N_{i}}{\partial t}+\sum^{d}_{\alpha=1}u^{\alpha}_{i}\dfrac{\partial N_{i}}{\partial x_{\alpha}} & =\underbrace{\dfrac{1}{2}\sum_{j,k,l\neq i}A^{kl}_{ij}\left(N_{k}N_{l}\text{\textminus}N_{i}N_{j}\right)}_{\equiv Q_{i}(N)}\label{eq:erttrererer}
\end{align}
where $\left\{ A^{kl}_{ij},i,j,k,l=1,\cdots,p\right\} \subset\R_{+}$
and $\vec{u}_{i}=\left(u^{\alpha}_{i}\right)^{d}_{\alpha=1}\in\R^{d}.$\\
The four-velocity Broadwell model in the plane is a particular case
of the system \eqref{eq:erttrererer} for which $d=2,$ $p=4$, $\overrightarrow{u_{1}}=c\left(1,0\right),$$\overrightarrow{u_{2}}=c\left(0,1\right),$
$\overrightarrow{u_{3}}=c\left(0,-1\right)$ and $\overrightarrow{u_{4}}=c\left(-1,0\right).$
The constants $A^{kl}_{ij}$ are given by 
\[
A^{23}_{14}=A^{32}_{14}=2cS,A^{kl}_{1j}=0\:\text{if}\:j\neq4,
\]
\[
A^{14}_{23}=A^{41}_{23}=2cS,A^{kl}_{2j}=0\:\text{if}\:j\neq3,
\]
\[
A^{14}_{32}=A^{41}_{32}=2cS,A^{kl}_{3j}=0\:\text{if}\:j\neq2,
\]
\[
A^{23}_{41}=A^{32}_{41}=2cS,A^{kl}_{4j}=0\:\text{if}\:j\neq1.
\]

$N_{i}\equiv N_{i}(t,x,y)\in\R$ and the kinetic equations are given
by: 
\begin{equation}
\begin{cases}
{\textstyle \dfrac{\p N_{1}}{\p t}+c\dfrac{\p N_{1}}{\p x}}=Q\left(N\right)\\
\\\dfrac{\p N_{2}}{\p t}+c\dfrac{\p N_{2}}{\p y}=-Q\left(N\right)\\
\\\dfrac{\p N_{3}}{\p t}-c\dfrac{\p N_{3}}{\p y}=-Q\left(N\right)\\
\\\dfrac{\p N_{4}}{\p t}-c\dfrac{\p N_{4}}{\p x}=Q\left(N\right)
\end{cases}\label{eq:koiqmn}
\end{equation}

\begin{equation}
Q\left(N\right)=2cS\left(N_{2}N_{3}-N_{1}N_{4}\right),\label{eq:jjdpoa}
\end{equation}
$c,S>0$ constants.

\subsection{The class of data}\label{subsec:The-class-of}
\begin{enumerate}
\item Let $\left[a_{1},b_{1}\right]\times\left[a_{2},b_{2}\right]\subset\R^{2}.$
We consider for $i=1,\cdots,4$ and a parameter $\tau\geq0,$ initial
data that are piecewise $C^{1}$functions $N^{\tau}_{i}:\left[a_{1};b_{1}\right]\times\left[a_{2};b_{2}\right]\rightarrow\R,$
$\left(x,y\right)\mapsto N^{\tau}_{i}\left(x,y\right)$ such that 
\begin{enumerate}
\item $N^{\tau}_{i}$ are continuous (thus bounded ) on $\left[a_{1};b_{1}\right]\times\left[a_{2};b_{2}\right];$
\item $\partial_{x}N^{\tau}_{1}$, $\partial_{y}N^{\tau}_{1}$ , $\partial_{x}N^{\tau}_{2}$,
$\partial_{y}N^{\tau}_{2}$, $\partial_{x}N^{\tau}_{3}$, $\partial_{y}N^{\tau}_{3}$,
$\partial_{x}N^{\tau}_{4}$, $\partial_{y}N^{\tau}_{4}$ are defined
in $\left]a_{1};b_{1}\right[\times\left]a_{2};b_{2}\right[$ except
possibly on finitely many lines with respective equations in the form
$x=k_{1},\cdots,x=k_{\alpha}$, $y=h_{1},\cdots,y=h_{\beta}$ where
$k_{1},\cdots,k_{\alpha}$ , $h_{1},\cdots,h_{\beta}$ are arbitrary
constant real numbers;
\item $\partial_{x}N^{\tau}_{i}$ and $\partial_{y}N^{\tau}_{i}$, $i=1,2,3,4$
are bounded.
\end{enumerate}
\end{enumerate}
\begin{enumerate}[resume]
\item Moreover, we consider boundary data that are also piecewise $C^{1}$
real-valued functions 
\[
N^{-}_{1}\equiv N^{-}_{1}\left(t,y\right),\left(t,y\right)\in\left[0;+\infty\right[\times\left[a_{2};b_{2}\right]
\]
\[
N^{-}_{2}\equiv N^{-}_{2}\left(t,x\right),\left(t,x\right)\in\left[0;+\infty\right[\times\left[a_{1};b_{1}\right]
\]
\[
N^{+}_{3}\equiv N^{+}_{3}\left(t,x\right),\left(t,x\right)\in\left[0;+\infty\right[\times\left[a_{1};b_{1}\right]
\]
\[
N^{+}_{4}\equiv N^{+}_{4}\left(t,y\right),\left(t,y\right)\in\left[0;+\infty\right[\times\left[a_{2};b_{2}\right]
\]
 such that 
\begin{enumerate}
\item $N^{-}_{1}$ and $N^{+}_{4}$ ( resp. $N^{-}_{2}$ and $N^{+}_{3}$)
are continuous and bounded on $\left[0;+\infty\right[\times\left[a_{2};b_{2}\right]$
(resp. $\left[0;+\infty\right[\times\left[a_{1};b_{1}\right]$); 
\item for every compact interval $\left[\tau;\tau'\right]\subset\left[0;+\infty\right[$,
the first order partial derivatives of $N^{-}_{1}$, $N^{+}_{4}$
(resp. $N^{-}_{2}$, $N^{+}_{3}$) are defined on $\left]\tau;\tau'\right[\times\left]a_{2};b_{2}\right[$
(resp. $\left]\tau;\tau'\right[\times\left]a_{1};b_{1}\right[$) except
on finitely many lines with respective equations in the form $t=r_{1},\cdots,t=r_{\mu}$
where $r_{1},\cdots,r_{\mu}$ are arbitrary constant real numbers;
\item the first order partial derivatives of $N^{-}_{i},$$i=1,2$ (resp.
$N^{+}_{i},$$i=3,4$) are bounded.
\end{enumerate}
\item Finally, all the initial and boundary data are assumed to be positive.
\end{enumerate}
We shall use the uniform norm $\|\cdot\|_{\infty}$ for the initial
and boundary data and their first-order partial derivatives. We put
\begin{equation}
\left\Vert \nabla N^{\tau}_{i}\right\Vert =\max\left\{ \|\partial_{x}N^{\tau}_{i}\|_{\infty},\|\partial_{y}N^{\tau}_{i}\|_{\infty}\right\} \label{iidod-1-1}
\end{equation}
\begin{equation}
\left\Vert \nabla N^{-}_{1}\right\Vert =\max\left\{ \|\partial_{t}N^{-}_{1}\|_{\infty},\|\partial_{y}N^{-}_{1}\|_{\infty}\right\} \label{iidod-2}
\end{equation}
\begin{equation}
\left\Vert \nabla N^{-}_{2}\right\Vert =\max\left\{ \|\partial_{t}N^{-}_{2}\|_{\infty},\|\partial_{x}N^{-}_{2}\|_{\infty}\right\} \label{iidod-3}
\end{equation}
\begin{equation}
\left\Vert \nabla N^{+}_{3}\right\Vert =\max\left\{ \|\partial_{t}N^{+}_{3}\|_{\infty},\|\partial_{x}N^{+}_{3}\|_{\infty}\right\} \label{iidod-4}
\end{equation}
\begin{equation}
\left\Vert \nabla N^{+}_{4}\right\Vert =\max\left\{ \|\partial_{t}N^{+}_{4}\|_{\infty},\|\partial_{y}N^{+}_{4}\|_{\infty}\right\} \label{iidod-4-1}
\end{equation}

\subsection{Initial-boundary value problems}

Let us give initial data $N^{0}_{i},$$i=1,\cdots,4$ and boundary
data $N^{-}_{1}$, $N^{-}_{2}$, $N^{+}_{3}$ and $N^{+}_{4}$ in
the class defined in Subsection \ref{subsec:The-class-of}. For a
given $\tau\geq0,$ let us also consider initial data $N^{\tau}_{i},$$i=1,\cdots,4$
, in the previous class.
\begin{enumerate}
\item We define the problem $\Sigma$ by the system \eqref{eq:koiqmn} on
$\left[0;+\infty\right[\times\left[a_{1},b_{1}\right]\times\left[a_{2},b_{2}\right]$
with the initial and boundary conditions \eqref{eq:lsos}-\eqref{eq:ikioi-1}
\end{enumerate}
\begin{align}
N_{i}\left(0,x,y\right)= & N^{0}_{i}\left(x,y\right),\;\left(x,y\right)\in\left[a_{1};b_{1}\right]\times\left[a_{2};b_{2}\right],i=1,\cdots,4\label{eq:lsos}\\
N_{1}\left(t,a_{1},y\right)= & N^{-}_{1}\left(t,y\right),\;\left(t,y\right)\in\left[0;+\infty\right[\times\left[a_{2};b_{2}\right]\label{eq:losso}\\
N_{2}\left(t,x,a_{2}\right)= & N^{-}_{2}\left(t,x\right),\;\left(t,x\right)\in\left[0;+\infty\right[\times\left[a_{1};b_{1}\right]\label{eq:lsoo}\\
N_{3}\left(t,x,b_{2}\right)= & N^{+}_{3}\left(t,x\right),\;\left(t,x\right)\in\left[0;+\infty\right[\times\left[a_{1};b_{1}\right]\label{eq:ikioi}\\
N_{4}\left(t,b_{1},y\right)= & N^{+}_{4}\left(t,y\right),\;\left(t,y\right)\in\left[0;+\infty\right[\times\left[a_{2};b_{2}\right]\label{eq:ikioi-1}
\end{align}
The initial and boundary data satisfy the compatibility conditions
\eqref{eq:kqiiq}-\eqref{eq:looqp-1}

\begin{align}
N^{0}_{1}\left(a_{1},y\right) & =N^{-}_{1}\left(0,y\right),\;y\in\left[a_{2};b_{2}\right]\label{eq:kqiiq}\\
N^{0}_{2}\left(x,a_{2}\right) & =N^{-}_{2}\left(0,x\right),\;x\in\left[a_{1};b_{1}\right]\label{eq:lqoop}\\
N^{0}_{3}\left(x,b_{2}\right) & =N^{+}_{3}\left(0,x\right),\;x\in\left[a_{1};b_{1}\right]\label{eq:loppq-1}\\
N^{0}_{4}\left(b_{1},y\right) & =N^{+}_{4}\left(0,y\right),\;y\in\left[a_{2};b_{2}\right]\label{eq:looqp-1}
\end{align}

\begin{enumerate}[resume]
\item For $\left[\tau,\tau'\right]\subset\R^{+},$ we define the problem
$\Sigma_{\tau,\tau'}$ by the system \eqref{eq:koiqmn} on $\left[\tau,\tau'\right]\times\left[a_{1},b_{1}\right]\times\left[a_{2},b_{2}\right]$
with the initial and boundary conditions \eqref{eq:lsos-1}-\eqref{eq:ikioi-1-1}
\begin{align}
N_{i}\left(\tau,x,y\right)= & N^{\tau}_{i}\left(x,y\right),\;\left(x,y\right)\in\left[a_{1};b_{1}\right]\times\left[a_{2};b_{2}\right],i=1,\cdots,4\label{eq:lsos-1}\\
N_{1}\left(t,a_{1},y\right)= & N^{-}_{1}\left(t,y\right),\;\left(t,y\right)\in\left[\tau,\tau'\right]\times\left[a_{2};b_{2}\right]\label{eq:losso-1}\\
N_{2}\left(t,x,a_{2}\right)= & N^{-}_{2}\left(t,x\right),\;\left(t,x\right)\in\left[\tau,\tau'\right]\times\left[a_{1};b_{1}\right]\label{eq:lsoo-1}\\
N_{3}\left(t,x,b_{2}\right)= & N^{+}_{3}\left(t,x\right),\;\left(t,x\right)\in\left[\tau,\tau'\right]\times\left[a_{1};b_{1}\right]\label{eq:ikioi-2}\\
N_{4}\left(t,b_{1},y\right)= & N^{+}_{4}\left(t,y\right),\;\left(t,y\right)\in\left[\tau,\tau'\right]\times\left[a_{2};b_{2}\right]\label{eq:ikioi-1-1}
\end{align}
and the compatibility conditions \eqref{eq:kqiiq-1}-\eqref{eq:looqp-1-1}
\begin{align}
N^{\tau}_{1}\left(a_{1},y\right) & =N^{-}_{1}\left(\tau,y\right),\;y\in\left[a_{2};b_{2}\right]\label{eq:kqiiq-1}\\
N^{\tau}_{2}\left(x,a_{2}\right) & =N^{-}_{2}\left(\tau,x\right),\;x\in\left[a_{1};b_{1}\right]\label{eq:lqoop-1}\\
N^{\tau}_{3}\left(x,b_{2}\right) & =N^{+}_{3}\left(\tau,x\right),\;x\in\left[a_{1};b_{1}\right]\label{eq:loppq-1-1}\\
N^{\tau}_{4}\left(b_{1},y\right) & =N^{+}_{4}\left(\tau,y\right),\;y\in\left[a_{2};b_{2}\right]\label{eq:looqp-1-1}
\end{align}
\end{enumerate}

\subsection{The class of solutions}\label{subsec:The-class-of-1}
\begin{defn}
\label{def:nshhsiz}For $\left[\tau,\tau'\right]\subset\R^{+},$ we
consider the class $E^{\tau,\tau'}$ of piecewise $C^{1}$ functions
\[
u:\left[\tau,\tau'\right]\times\left[a_{1},b_{1}\right]\times\left[a_{2},b_{2}\right]\rightarrow\R,\left(t,x,y\right)\mapsto u\left(t,x,y\right)
\]
 such that
\end{defn}

\begin{enumerate}
\item $u$ is continuous on $\left[\tau,\tau'\right]\times\left[a_{1},b_{1}\right]\times\left[a_{2},b_{2}\right]$;
\item for every $u\in E^{\tau,\tau'},$ there exist finitely many planes
with respective equations in the form $x-ct=k_{1},\cdots,x-ct=k_{\alpha}$,$y-ct=h_{1},\cdots,y-ct=h_{\beta}$,
$y+ct=l_{1},\cdots,y+ct=l_{\gamma}$, $x+ct=m_{1},\cdots,x+ct=m_{\delta}$
where $k_{1},\cdots,k_{\alpha}$ , $h_{1},\cdots,h_{\beta}$, $l_{1},\cdots,l_{\gamma}$,
$m_{1},\cdots,m_{\delta}$ are arbitrary constant real numbers, such
that $\partial_{t}u,\partial_{x}u$ and $\partial_{y}u$ are defined
on $\left]\tau;\tau'\right[\times\left]a_{1};b_{1}\right[\times\left]a_{2};b_{2}\right[$
except possibly on these planes;
\item $\partial_{t}u,\partial_{x}u$ and $\partial_{y}u$ are bounded.
\end{enumerate}
For $u\in E^{\tau,\tau'},$ such that Condition 2 in Definition \ref{def:nshhsiz}
is satisfied for given values of the constants $k_{1},\cdots,k_{\alpha}$
, $h_{1},\cdots,h_{\beta}$, $l_{1},\cdots,l_{\gamma}$, $m_{1},\cdots,m_{\delta}$,
we put 
\begin{equation}
\|u\|_{\infty}\equiv\sup\left\{ \left|u\left(t,x,y\right)\right|:\left(t,x,y\right)\in\left[\tau,\tau'\right]\times\left[a_{1},b_{1}\right]\times\left[a_{2},b_{2}\right]\right\} \label{odpele}
\end{equation}
\begin{multline}
\|\partial_{t}u\|_{\infty}\equiv\sup\left\{ \right.\left|\partial_{t}u\left(t,x,y\right)\right|:\left(t,x,y\right)\in\left]\tau;\tau'\right[\times\left]a_{1};b_{1}\right[\times\left]a_{2};b_{2}\right[,\\
x-ct\neq k_{1},\cdots,x-ct\neq k_{\alpha},y-ct\neq h_{1},\cdots,y-ct\neq h_{\beta},\\
y+ct\neq l_{1},\cdots,y+ct\neq l_{\gamma},x+ct\neq m_{1},\cdots,x+ct\neq m_{\delta}\left.\right\} \label{ldoeje}
\end{multline}
and similarly $\|\partial_{x}u\|_{\infty},\|\partial_{y}u\|_{\infty}.$ 

For a $4$-tuple $N=\left(N_{i}\right)^{4}_{i=1}\in\left(E^{\tau,\tau'}\right)^{4}$,
we put 
\begin{equation}
\left\Vert N\right\Vert =\max_{1\leq i\leq4}\|N_{i}\|_{\infty}\label{oskie}
\end{equation}
and for every $i=1,2,3,4,$
\begin{equation}
\left\Vert \nabla N_{i}\right\Vert =\max\left\{ \|\partial_{t}N_{i}\|_{\infty},\|\partial_{x}N_{i}\|_{\infty},\|\partial_{y}N_{i}\|_{\infty}\right\} .\label{iidod}
\end{equation}

\begin{rem}
\label{rem:hsyys}Let $N=\left(N_{i}\right)^{4}_{i=1}\in\left(E^{\tau,\tau'}\right)^{4}.$
From Condition 2 in Definition \ref{def:nshhsiz}, for $i=1,2,3,4$,
$\partial_{t}N_{i},\partial_{x}N_{i}$ and $\partial_{y}N_{i}$ are
defined except possibly on finitely many planes with respective equations
in the form $x-ct=k_{1},\cdots,x-ct=k_{\alpha}$ , $k_{1},\cdots,k_{\alpha}$
constants. Therefore for $i=1,2,3,4$, and every $t_{0}\in\left[\tau,\tau'\right],$
the function $N_{i}\left(t_{0},\cdot,\cdot\right):\left[a_{1};b_{1}\right]\times\left[a_{2};b_{2}\right]\rightarrow\R,$
$\left(x,y\right)\mapsto N_{i}\left(t_{0},x,y\right)$ is in the class
of the initial data $N^{\tau}_{i}$ as defined in Subsection \ref{subsec:The-class-of}(Condition
1-b) since $\partial_{x}N_{i}\left(t_{0},\cdot,\cdot\right)$, $\partial_{y}N_{i}\left(t_{0},\cdot,\cdot\right)$
are defined in $\left]a_{1};b_{1}\right[\times\left]a_{2};b_{2}\right[$
except possibly on finitely many lines with respective equations in
the form $x=ct_{0}+k_{1},\cdots,x=ct_{0}+k_{\alpha},$$x=-ct_{0}+m_{1},\cdots,x=-ct_{0}+m_{\delta}$,
$y=ct_{0}+h_{1},\cdots,y=ct_{0}+h_{\beta}$, $y=-ct_{0}+l_{1},\cdots,y=-ct_{0}+l_{\gamma}$. 
\end{rem}

Let us set for $\left(t,x,y\right)\in\left[\tau,\tau'\right]\times\left[a_{1},b_{1}\right]\times\left[a_{2},b_{2}\right],$
\begin{align}
\overline{N^{\tau}_{1}}\left(t,x,y\right) & \equiv N^{\tau}_{1}\left(x-c\left(t-\tau\right),y\right)\label{eq:kqiiq-1-1}\\
\overline{N^{\tau}_{2}}\left(t,x,y\right) & \equiv N^{\tau}_{2}\left(x,y-c\left(t-\tau\right)\right)\label{eq:lqoop-1-1}\\
\overline{N^{\tau}_{3}}\left(t,x,y\right) & \equiv N^{\tau}_{3}\left(x,y+c\left(t-\tau\right)\right)\label{eq:loppq-1-1-1}\\
\overline{N^{\tau}_{4}}\left(t,x,y\right) & \equiv N^{\tau}_{4}\left(x+c\left(t-\tau\right),y\right)\label{eq:looqp-1-1-1}
\end{align}
and for $\left(t,x,y\right)\in\left[0;+\infty\right[\times\left[a_{1},b_{1}\right]\times\left[a_{2},b_{2}\right],$
when they are defined 
\begin{align}
\overline{N^{-}_{1}}\left(t,x,y\right) & \equiv N^{-}_{1}\left(t-\frac{x-a_{1}}{c},y\right)\label{eq:kqiiq-1-1-1}\\
\overline{N^{-}_{2}}\left(t,x,y\right) & \equiv N^{-}_{2}\left(t-\frac{y-a_{2}}{c},x\right)\label{eq:lqoop-1-1-1}\\
\overline{N^{+}_{3}}\left(t,x,y\right) & \equiv N^{+}_{3}\left(t-\frac{b_{2}-y}{c},x\right)\label{eq:loppq-1-1-1-1}\\
\overline{N^{+}_{4}}\left(t,x,y\right) & \equiv N^{+}_{4}\left(t-\frac{b_{1}-x}{c},y\right).\label{eq:looqp-1-1-1-1}
\end{align}

\begin{rem}
The function $\overline{N^{\tau}_{1}}$ is in the class $E^{\tau,\tau'}.$
In fact, from Condition 1-b) in Subsection \ref{subsec:The-class-of},
$\partial_{x}N^{\tau}_{1}$, $\partial_{y}N^{\tau}_{1}$ are defined
in $\left]a_{1};b_{1}\right[\times\left]a_{2};b_{2}\right[$ except
possibly for finitely many values $x=k_{1},\cdots,x=k_{\alpha}$ of
$x$ . Therefore, $\partial_{x}\overline{N^{\tau}_{1}}$, $\partial_{x}\overline{N^{\tau}_{1}}$
and $\partial_{y}\overline{N^{\tau}_{1}}$ are defined except possibly
for $x-c\left(t-\tau\right)=k_{1},\cdots,x-c\left(t-\tau\right)=k_{\alpha},$
that is except possibly for $x-ct=k_{1}-c\tau,$$\cdots,$$x-ct=k_{\alpha}-c\tau$,
that is Condition 2 above holds. Similarly, $\overline{N^{-}_{1}},$
$\overline{N^{\tau}_{2}},$ $\overline{N^{-}_{2}}$, $\overline{N^{\tau}_{3}},$$\overline{N^{+}_{3}}$
, $\overline{N^{\tau}_{4}}$ and $\overline{N^{+}_{4}}$ are in the
class $E^{\tau,\tau'}.$
\end{rem}

Our aim is to prove:
\begin{itemize}
\item the existence and uniqueness of a positive piecewise $C^{1}$ solution
$N^{\tau,\tau'}=\left(N^{\tau,\tau'}_{i}\right)^{4}_{i=1}$ such that
$N^{\tau,\tau'}\in\left(E^{\tau,\tau'}\right)^{4}$ to the problem
$\Sigma_{\tau,\tau'}$ and 
\item the existence of a positive piecewise $C^{1}$ solution $N=\left(N_{i}\right)^{4}_{i=1}$
to the problem $\Sigma$ such that
\begin{itemize}
\item there exists $T=+\infty$ or $T<+\infty$ such that the domain of
$N$ is $\left[0,T\right[\times\left[a_{1},b_{1}\right]\times\left[a_{2},b_{2}\right]$ 
\item $N$ is continuous on $\left[0,T\right[\times\left[a_{1},b_{1}\right]\times\left[a_{2},b_{2}\right]$
and 
\item for every $\left[\tau,\tau'\right]\subset\left[0,T\right[,$ the restriction
of $N$ to $\left[\tau,\tau'\right]\times\left[a_{1},b_{1}\right]\times\left[a_{2},b_{2}\right]$
is in $\left(E^{\tau,\tau'}\right)^{4}.$ 
\end{itemize}
\end{itemize}

\subsection{Main results}
\begin{notation}
\label{nota:jjskks}Let $T=+\infty$ or $T<+\infty$. For a function
$N=\left(N_{i}\right)^{4}_{i=1}\in C\left(\left[0,T\right[\times\left[a_{1},b_{1}\right]\times\left[a_{2},b_{2}\right],\R\right)^{4}$
such that for every $i=1,2,3,4,$ $N_{i}$ is bounded, we put 
\begin{equation}
\left\Vert N\right\Vert \equiv\max_{1\leq i\leq4}\left\Vert N_{i}\right\Vert _{\infty}.\label{isnsnj}
\end{equation}
\end{notation}

Let us set, for a fixed $\left[\tau,\tau'\right]$ 
\begin{equation}
q\equiv\max_{1\leq i\leq4}\Biggl\{\left\Vert N^{\tau}_{i}\right\Vert _{\infty},\left\Vert N^{-}_{1}\right\Vert _{\infty},\left\Vert N^{-}_{2}\right\Vert _{\infty},\left\Vert N^{+}_{3}\right\Vert _{\infty},\left\Vert N^{+}_{4}\right\Vert _{\infty}\Biggr\}\label{bghghf}
\end{equation}
and 
\begin{equation}
q'\equiv\max_{1\leq i\leq4}\left\{ \left\Vert \nabla N^{\tau}_{i}\right\Vert ,\left\Vert \nabla N^{-}_{1}\right\Vert ,\left\Vert \nabla N^{-}_{2}\right\Vert ,\left\Vert \nabla N^{+}_{3}\right\Vert ,\left\Vert \nabla N^{+}_{4}\right\Vert \right\} \label{oloo-1}
\end{equation}

We prove in the sequel the following results 
\begin{thm}
\label{thm:uniqueness}There exists at most one maximal piecewise
$C^{1}$ solution $N$ to the problem $\Sigma$ such that $N$ is
continuous and for every $\left[\tau,\tau'\right]$, the restriction
of $N$ to $\left[\tau,\tau'\right]\times\left[a_{1},b_{1}\right]\times\left[a_{2},b_{2}\right]$
is in $\left(E^{\tau,\tau'}\right)^{4}.$
\end{thm}

\begin{thm}
\label{thm:If-the-data}The solution to the problem $\Sigma$ when
it exists, is positive. 
\end{thm}

\begin{thm}
\label{thm:Suppose-.-Then} Assume that $\tau'-\tau<\frac{1+2c}{\left(1+c\right)^{2}}\cdot\dfrac{1}{16cSq}.$
Then the problem $\Sigma_{\tau,\tau'}$ admits an unique positive
piecewise $C^{1}$ solution $N^{\tau,\tau'}=\left(N^{\tau,\tau'}_{i}\right)^{4}_{i=1}\in\left(E^{\tau,\tau'}\right)^{4}$
such that 
\begin{equation}
\left\Vert N^{\tau,\tau'}\right\Vert \leq\dfrac{1-\sqrt{1-16cS\left(\tau'-\tau\right)q}}{8cS\left(\tau'-\tau\right)}\equiv R_{\tau,\tau',q}\label{skeut}
\end{equation}
Moreover \textup{
\begin{equation}
\tau'-\tau<\frac{1}{8cS\left(1+c\right)R_{\tau,\tau',q}}\label{ksisiiz}
\end{equation}
and 
\begin{equation}
\max_{1\leq i\leq4}\left\Vert \nabla N^{\tau,\tau'}_{i}\right\Vert \leq\dfrac{\left(1+\frac{1}{c}\right)8cSR^{2}_{\tau,\tau',q}+\left(1+c+\frac{1}{c}\right)q'}{1-8cS\left(1+c\right)\left(\tau'-\tau\right)R_{\tau,\tau',q}}\label{sujr}
\end{equation}
}
\end{thm}

\begin{cor}
\label{cor:lsoosz} Let $R_{0}$ be an arbitrary real constant such
that $q<R_{0}$. Assume 
\begin{equation}
\tau'-\tau=\min\left\{ \frac{1}{8cSR_{0}+\frac{\left(1+c\right)^{2}}{1+2c}16cSq},\frac{R_{0}-q}{8cSR^{2}_{0}}\right\} .\label{snshy}
\end{equation}
Then the problem $\Sigma_{\tau,\tau'}$ admits an unique positive
piecewise $C^{1}$ solution $N^{\tau,\tau'}=\left(N^{\tau,\tau'}_{i}\right)^{4}_{i=1}\in\left(E^{\tau,\tau'}\right)^{4}$
and \textup{$N^{\tau,\tau'}$ satisfies }
\begin{equation}
\left\Vert N^{\tau,\tau'}\right\Vert <R_{0}.\label{skeut-1}
\end{equation}
\end{cor}

\begin{thm}
\label{thm:ksiisiiq}The problem $\Sigma$ admits an unique global,
bounded and positive piecewise $C^{1}$ solution $N=\left(N_{i}\right)^{4}_{i=1}$such
that 

(i) $N$ is continuous on $\left[0;+\infty\right[\times\left[a_{1},b_{1}\right]\times\left[a_{2},b_{2}\right]$, 

(ii) for every $\left[\tau,\tau'\right]\subset\R^{+},$ the restriction
of $N$ to $\left[\tau,\tau'\right]\times\left[a_{1},b_{1}\right]\times\left[a_{2},b_{2}\right]$
is in $\left(E^{\tau,\tau'}\right)^{4}.$\\
Moreover $N$ satisfies 
\begin{equation}
\left\Vert N\right\Vert =\max_{1\leq i\leq4}\Biggl\{\left\Vert N^{0}_{i}\right\Vert _{\infty},\left\Vert N^{-}_{1}\right\Vert _{\infty},\left\Vert N^{-}_{2}\right\Vert _{\infty},\left\Vert N^{+}_{3}\right\Vert _{\infty},\left\Vert N^{+}_{4}\right\Vert _{\infty}\Biggr\}.\label{neeue}
\end{equation}
\end{thm}

\section{Reformulation of one initial-boundary value problem as a fixed-point
problem}\label{kooollooaooaiei}

Let us put for a fixed $\left[\tau,\tau'\right],$ $\mathscr{P}\equiv\left[\tau,\tau'\right]\times\left[a_{1};b_{1}\right]\times\left[a_{2};b_{2}\right].$
Let us consider the problem $\Sigma_{\tau,\tau'}.$\\
 Suppose that the derivatives $\partial_{x}N^{\tau}_{1}$, $\partial_{y}N^{\tau}_{1}$
, $\partial_{x}N^{\tau}_{2}$, $\partial_{y}N^{\tau}_{2}$, $\partial_{x}N^{\tau}_{3}$,
$\partial_{y}N^{\tau}_{3}$, $\partial_{x}N^{\tau}_{4}$, $\partial_{y}N^{\tau}_{4}$
of the initial data, are defined in $\left]a_{1};b_{1}\right[\times\left]a_{2};b_{2}\right[$
except possibly on finitely many lines with respective equations in
the form $x=k_{1},\cdots,x=k_{\alpha}$, $y=h_{1},\cdots,y=h_{\beta}$
where $k_{1},\cdots,k_{\alpha}$ , $h_{1},\cdots,h_{\beta}$ are given
constant real numbers and that the first order partial derivatives
of $N^{-}_{1}$, $N^{+}_{4}$ (resp. $N^{-}_{2}$, $N^{+}_{3}$) are
defined on $\left]\tau;\tau'\right[\times\left]a_{2};b_{2}\right[$
(resp. $\left]\tau;\tau'\right[\times\left]a_{1};b_{1}\right[$) except
on finitely many lines with respective equations in the form $t=r_{1},\cdots,t=r_{\mu}$
where $r_{1},\cdots,r_{\mu}$ are given constant real numbers. 
\begin{defn}
Let us consider the finite family denoted $\mathscr{F}$ of the following
planes with respective equations \eqref{lsopzm} and \eqref{lddo}:
\begin{multline}
x-c\left(t-\tau\right)=k_{1},\cdots,x-c\left(t-\tau\right)=k_{\alpha},\:\text{and}\:t-\frac{x-a_{1}}{c}=r_{1},\cdots,t-\frac{x-a_{1}}{c}=r_{\mu},\\
y-c\left(t-\tau\right)=h_{1},\cdots,y-c\left(t-\tau\right)=h_{\beta},\:\text{and}\:t-\frac{y-a_{2}}{c}=r_{1},\cdots,t-\frac{y-a_{2}}{c}=r_{\mu},\\
y+c\left(t-\tau\right)=h_{1},\cdots,y+c\left(t-\tau\right)=h_{\beta},\:\text{and}\:t-\frac{b_{2}-y}{c}=r_{1},\cdots,t-\frac{b_{2}-y}{c}=r_{\mu}\\
x+c\left(t-\tau\right)=k_{1},\cdots,x+c\left(t-\tau\right)=k_{\alpha},\:\text{and}\:t-\frac{b_{1}-x}{c}=r_{1},\cdots,t-\frac{b_{1}-x}{c}=r_{\mu}\label{lsopzm}
\end{multline}
\begin{equation}
x-c\left(t-\tau\right)=a_{1},y-c\left(t-\tau\right)=a_{2},y+c\left(t-\tau\right)=b_{2}\:\text{and}\:x+c\left(t-\tau\right)=b_{1}.\label{lddo}
\end{equation}
\end{defn}

\begin{rem}
$\mathscr{F}$ constitute a finite family of planes with respective
equations in the form 
\begin{multline}
x-ct=s_{1},\cdots,x-ct=s_{p},y-ct=s'_{1},\cdots,y-ct=s'_{p'},\\
y+ct=v_{1},\cdots,y+ct=v_{q},x+ct=v'_{1},\cdots,x+ct=v'_{q'}\label{kdioer}
\end{multline}
 (where $s_{1},\cdots,s_{p}$ , $s'_{1},\cdots,s'_{p'}$, $v_{1},\cdots,v_{q}$,
$v'_{1},\cdots,v'_{q'}$ are arbitrary constant real numbers).
\end{rem}

\begin{defn}
\label{def:jsusooa}Let us denote by $\mathscr{\mathscr{O}_{\mathscr{F}}}$
the open set consisting of points of $\mathring{\mathscr{P}}$ that
are in none of the planes of the family $\mathscr{F}:$
\begin{equation}
\mathscr{O}_{\mathscr{F}}\equiv\mathring{\mathscr{P}}\setminus\left(\cup\mathscr{F}\right).\label{kssos}
\end{equation}
For any given finite family $\mathscr{G}$ of planes with respective
equations in the form \eqref{kdioer}, let us denote by $\mathscr{\mathscr{G}\cup\mathscr{F}}$
the finite family of planes that are in $\mathscr{F}$ or in $\mathscr{G}$
. Let us form the open set $\mathscr{O}_{\mathscr{G}\cup\mathscr{F}}$
consisting of points of $\mathring{\mathscr{P}}$ that are in none
of the planes of the family $\mathscr{\mathscr{G}\cup\mathscr{F}}.$
\begin{equation}
\mathscr{O}_{\mathscr{\mathscr{G}\cup\mathscr{F}}}\equiv\mathring{\mathscr{P}}\setminus\left(\cup\left(\mathscr{\mathscr{G}\cup\mathscr{F}}\right)\right).\label{kssos-1}
\end{equation}
\end{defn}

\begin{rem}
\label{rem:sjsusuus}It is immediate that the functions $\overline{N^{\tau}_{1}}$,
$\overline{N^{-}_{1}}$, $\overline{N^{\tau}_{2}}$, $\overline{N^{-}_{2}}$,
$\overline{N^{\tau}_{3}}$ , $\overline{N^{+}_{3}}$, $\overline{N^{\tau}_{4}}$
and $\overline{N^{+}_{4}}$ (\eqref{eq:kqiiq-1-1} - \eqref{eq:looqp-1-1-1-1})
are continuous on $\mathscr{P}$ and that their first order derivatives
are well defined, continuous and bounded on $\mathscr{O}_{\mathscr{\mathscr{\mathscr{G}\cup\mathscr{F}}}}.$
\\
\end{rem}

Now let $M=\left(M_{i}\right)^{4}_{i=1}\in C\left(\mathscr{P},\R\right)^{4}$
such that for every $i=1,2,3,4,$ $\partial_{t}M_{i},\partial_{x}M_{i}$,
$\partial_{y}M_{i}\in CB\left(\mathscr{O}_{\mathscr{\mathscr{\mathscr{G}\cup\mathscr{F}}}},\R\right)$
where $C\left(\mathscr{P},\R\right)$ (resp. $CB\left(\mathscr{O}_{\mathscr{\mathscr{\mathscr{G}\cup\mathscr{F}}}},\R\right)$)
is the space of continuous (resp. continuous and bounded) functions
$u:\mathscr{P}\mapsto\R$ (resp. $u:\mathscr{O}_{\mathscr{\mathscr{\mathscr{G}\cup\mathscr{F}}}}\mapsto\R$).
Consider the following decoupled linear system $\left(\Sigma^{M}_{\tau,\tau'}\right)$
defined by \eqref{eq:ssdffz-1-1-1-1}-\eqref{eq:qqqqsa-1-1-1-1} 
\begin{align}
{\textstyle \dfrac{\p N_{1}}{\p t}+c\dfrac{\p N_{1}}{\p x}} & =Q(M),\;\left(t,x,y\right)\in\mathscr{O}_{\mathscr{\mathscr{\mathscr{G}\cup\mathscr{F}}}}\label{eq:ssdffz-1-1-1-1}\\
\dfrac{\p N_{2}}{\p t}+c\dfrac{\p N_{2}}{\p y} & =-Q(M),\;\left(t,x,y\right)\in\mathscr{O}_{\mathscr{\mathscr{\mathscr{G}\cup\mathscr{F}}}}\\
\dfrac{\p N_{3}}{\p t}-c\dfrac{\p N_{3}}{\p y} & =-Q(M),\;\left(t,x,y\right)\in\mathscr{O}_{\mathscr{\mathscr{\mathscr{G}\cup\mathscr{F}}}}\\
\dfrac{\p N_{4}}{\p t}-c\dfrac{\p N_{4}}{\p x} & =Q(M),\;\left(t,x,y\right)\in\mathscr{O}_{\mathscr{\mathscr{\mathscr{G}\cup\mathscr{F}}}}\label{eq:qqqqsa-1-1-1-1}
\end{align}
with the conditions \eqref{eq:lsos-1}-\eqref{eq:ikioi-1-1}. \\
$\left(\Sigma^{M}_{\tau,\tau'}\right)$ admits an unique solution
$\mathcal{T}\left(M\right)=\left(\mathcal{T}_{i}\left(M\right)\right)^{4}_{i=1}$
such that $\mathcal{T}\left(M\right)\in C\left(\mathscr{P},\R\right)^{4}$
and for every $i=1,2,3,4,$ $\partial_{t}\mathcal{T}_{i}\left(M\right),$$\partial_{x}\mathcal{T}_{i}\left(M\right),$$\partial_{y}\mathcal{T}_{i}\left(M\right)\in CB\left(\mathscr{O}_{\mathscr{\mathscr{\mathscr{G}\cup\mathscr{F}}}},\R\right)$
\\
 $\mathcal{T}\left(M\right)$ is defined as followed: for an inequality
$f\left(t,x,y\right)\leq0$ defined on $\mathscr{P}$, let $\mathbb{I}_{f\left(t,x,y\right)\leq0}$
denote the identity function of the set $\left\{ \left(t,x,y\right)\in\mathscr{P}:f\left(t,x,y\right)\leq0\right\} ;$
then (using \eqref{eq:kqiiq-1-1}-\eqref{eq:looqp-1-1-1-1}) 
\begin{multline}
\mathcal{T}_{1}\left(M\right)\left(t,x,y\right)=\mathcal{T}^{A}_{1}\left(M\right)\left(t,x,y\right)\cdot\mathbb{I}_{x-c\left(t-\tau\right)\geq a_{1}}\left(t,x,y\right)+\\
\mathcal{T}^{B}_{1}\left(M\right)\left(t,x,y\right)\cdot\mathbb{I}_{x-c\left(t-\tau\right)\leq a_{1}}\left(t,x,y\right)\label{aoalal-2}
\end{multline}
where 

\begin{align}
\mathcal{T}^{A}_{1}\left(M\right)\left(t,x,y\right) & ={\displaystyle \int^{t}_{\tau}}Q\left(M\right)\left(s,x+c\left(s-t\right),y\right)ds+\overline{N^{\tau}_{1}}\left(t,x,y\right)\label{eq:olole-1-3}
\end{align}

\begin{align}
\mathcal{T}^{B}_{1}\left(M\right)\left(t,x,y\right) & ={\displaystyle \int^{t}_{t-\frac{1}{c}x+\frac{a_{1}}{c}}}Q\left(M\right)\left(s,x+c\left(s-t\right),y\right)ds+\overline{N^{-}_{1}}\left(t,x,y\right)\label{eq:olole-1-1-2}
\end{align}
\begin{multline}
\mathcal{T}_{2}\left(M\right)\left(t,x,y\right)=\mathcal{T}^{A}_{2}\left(M\right)\left(t,x,y\right)\cdot\mathbb{I}_{y-c\left(t-\tau\right)\geq a_{2}}\left(t,x,y\right)+\\
\mathcal{T}^{B}_{2}\left(M\right)\left(t,x,y\right)\cdot\mathbb{I}_{y-c\left(t-\tau\right)\leq a_{2}}\left(t,x,y\right)\label{aoalal-1-2}
\end{multline}

where 
\begin{align}
\mathcal{T}^{A}_{2}\left(M\right)\left(t,x,y\right) & ={\textstyle {\displaystyle \int^{t}_{\tau}}-Q\left(M\right)}\left(s,x,y+c\left(s-t\right)\right)ds+{\textstyle \overline{N^{\tau}_{2}}\left(t,x,y\right)}\label{eq:olole-1-1-1-1-1-1-1}
\end{align}

\begin{align}
\mathcal{T}^{B}_{2}\left(M\right)\left(t,x,y\right) & ={\textstyle {\displaystyle \int^{t}_{t-\frac{1}{c}y+\frac{a_{2}}{c}}}}-Q\left(M\right)\left(s,x,y+c\left(s-t\right)\right)ds+\overline{N^{-}_{2}}\left(t,x,y\right)\label{eq:olole-1-1-1-1-2-1}
\end{align}

\begin{multline}
\mathcal{T}_{3}\left(M\right)\left(t,x,y\right)=\mathcal{T}^{A}_{3}\left(M\right)\left(t,x,y\right)\cdot\mathbb{I}_{y+c\left(t-\tau\right)\leq b_{2}}\left(t,x,y\right)+\\
\mathcal{T}^{B}_{3}\left(M\right)\left(t,x,y\right)\cdot\mathbb{I}_{y+c\left(t-\tau\right)\geq b_{2}}\left(t,x,y\right)+\label{aoalal-1-1-3}
\end{multline}
where 
\begin{align}
\mathcal{T}^{A}_{3}\left(M\right)\left(t,x,y\right) & ={\displaystyle \int^{t}_{\tau}}-Q\left(M\right)\left(s,x,y-c\left(s-t\right)\right)ds+\overline{N^{\tau}_{3}}\left(t,x,y\right)\label{eq:olole-1-2-2}
\end{align}
\begin{align}
\mathcal{T}^{B}_{3}\left(M\right)\left(t,x,y\right) & ={\displaystyle \int^{t}_{t+\frac{1}{c}y-\frac{b_{2}}{c}}}-Q\left(M\right)\left(s,x,y-c\left(s-t\right)\right)ds+\overline{N^{+}_{3}}\left(t,x,y\right)\label{eq:olole-1-1-1-1}
\end{align}

\begin{multline}
\mathcal{T}_{4}\left(M\right)\left(t,x,y\right)=\mathcal{T}^{A}_{4}\left(M\right)\left(t,x,y\right)\cdot\mathbb{I}_{x+c\left(t-\tau\right)\leq b_{1}}\left(t,x,y\right)+\\
\mathcal{T}^{B}_{4}\left(M\right)\left(t,x,y\right)\cdot\mathbb{I}_{x+c\left(t-\tau\right)\geq b_{1}}\left(t,x,y\right)\label{aoalal-1-1-1-1}
\end{multline}
where 
\begin{align}
\mathcal{T}^{A}_{4}\left(M\right)\left(t,x,y\right) & ={\displaystyle \int^{t}_{\tau}}Q\left(M\right)\left(s,x-c\left(s-t\right),y\right)ds+\overline{N^{\tau}_{4}}\left(t,x,y\right)\label{eq:olole-1-2-1-2}
\end{align}
\begin{align}
\mathcal{T}^{B}_{4}\left(M\right)\left(t,x,y\right) & =\int^{t}_{t+\frac{1}{c}x-\frac{b_{1}}{c}}Q\left(M\right)\left(s,x-c\left(s-t\right),y\right)ds+\overline{N^{+}_{4}}\left(t,x,y\right).\label{eq:olole-1-1-1-2-2}
\end{align}

\begin{rem}
\label{rem:ssskzkz}We can easily verify that in the above formulas
defining the components of $\mathcal{T}\left(M\right),$ the partial
derivatives of the integrals are well defined. In fact, let $\left(t,x,y\right)\in\mathscr{\mathscr{O}_{\mathscr{\mathscr{\mathscr{G}\cup\mathscr{F}}}}}$;
the equations of the planes in $\mathscr{\mathscr{G}\cup\mathscr{F}}$
are respectively in the form \eqref{kdioer} for given values of $s_{1},\cdots,s_{p}$
, $s'_{1},\cdots,s'_{p'}$, $v_{1},\cdots,v_{q}$, $v'_{1},\cdots,v'_{q'}$
; hence $\left(t,x,y\right)$ verifies 
\[
x-ct\neq s_{1},\cdots,x-ct\neq s_{p},y-ct\neq s'_{1},\cdots,y-ct\neq s'_{p'},
\]
\[
y+ct\neq v_{1},\cdots,y+ct\neq v_{q},x+ct\neq v'_{1},\cdots,x+ct\neq v'_{q'}.
\]
Consider the expressions \eqref{eq:olole-1-3} and \eqref{eq:olole-1-1-2}
of $\mathcal{T}^{A}_{1}\left(M\right)\left(t,x,y\right)$ and $\mathcal{T}^{B}_{1}\left(M\right)\left(t,x,y\right).$
For all $s\in\R,$ the point $\left(t',x',y'\right)\equiv\left(s,x+c\left(s-t\right),y\right)$
satisfies $x'-ct'=x-ct.$ Hence 
\[
x'-ct'\neq s_{1},\cdots,x'-ct'\neq s_{p}.
\]
Each one of the equations $y'-ct'=cte,$ $y'+ct'=cte$ and $x'+ct'=cte$
where $cte$ is an arbitrary constant real number, yields an unique
value for $s.$ Hence the integrands in the formula \eqref{eq:olole-1-3}
and \eqref{eq:olole-1-1-2} are functions of the variable $s$ that
are derivable except possibly at finitely many points. Since these
integrands are continuous in $\left(s,t,x,y\right)$ and have all
first order partial derivatives with respect to $t,x,y$ that are
continuous, it follows that the partial derivatives of the integrals
in $\mathcal{T}_{1}\left(M\right)$ are well defined at the point
$\left(t,x,y\right)\in\mathscr{\mathscr{O}_{\mathscr{\mathscr{\mathscr{G}\cup\mathscr{F}}}}},$
and that we can differentiate under the integral sign, the integrands
in $\mathcal{T}_{1}\left(M\right).$ Hence, we obtain immediately
that $\dfrac{\partial\mathcal{T}_{1}\left(M\right)}{\partial t},\dfrac{\partial\mathcal{T}_{1}\left(M\right)}{\partial x},\dfrac{\partial\mathcal{T}_{1}\left(M\right)}{\partial y}$
$\in CB\left(\mathscr{O},\R\right).$ We similarly obtain from \ref{aoalal-1-2}-\ref{eq:olole-1-1-1-2-2}
that $\dfrac{\partial\mathcal{T}_{i}\left(M\right)}{\partial t},\dfrac{\partial\mathcal{T}_{i}\left(M\right)}{\partial x},\dfrac{\partial\mathcal{T}_{i}\left(M\right)}{\partial y},$$\left(i=2,3,4\right)$
are well defined on $\mathscr{\mathscr{O}_{\mathscr{\mathscr{\mathscr{G}\cup\mathscr{F}}}}}$
and are in $CB\left(\mathscr{\mathscr{O}_{\mathscr{\mathscr{\mathscr{G}\cup\mathscr{F}}}}},\R\right).$
\end{rem}

\begin{rem}
We can also easily verify that due to the compatibility conditions
\eqref{eq:kqiiq-1}-\eqref{eq:looqp-1-1}, the expressions \eqref{eq:olole-1-3}
and \eqref{eq:olole-1-1-2} coincide on the plane with equation $x-c\left(t-\tau\right)=a_{1}$
and similarly, the expressions \eqref{eq:olole-1-1-1-1-1-1-1} and
\eqref{eq:olole-1-1-1-1-2-1} on the plane $y-c\left(t-\tau\right)=a_{2}$,
the expressions \eqref{eq:olole-1-2-2} and \eqref{eq:olole-1-1-1-1}
on the plane $y+c\left(t-\tau\right)=b_{2}$, and the expressions
\eqref{eq:olole-1-2-1-2} and \eqref{eq:olole-1-1-1-2-2} on the plane
$x+c\left(t-\tau\right)=b_{1}.$ In fact for $x-c\left(t-\tau\right)=a_{1},$
we have $t-\frac{1}{c}x+\frac{a_{1}}{c}=\tau,$ hence from \eqref{eq:olole-1-3}
and \eqref{eq:kqiiq-1-1}, we successively have 
\[
\mathcal{T}^{A}_{1}\left(M\right)\left(t,x,y\right)={\displaystyle \int^{t}_{\tau}}Q\left(M\right)\left(s,x+c\left(s-t\right),y\right)ds+N^{\tau}_{1}\left(x-c\left(t-\tau\right),y\right)
\]
\[
\mathcal{T}^{A}_{1}\left(M\right)\left(t,x,y\right)={\displaystyle \int^{t}_{t-\frac{1}{c}x+\frac{a_{1}}{c}}}Q\left(M\right)\left(s,x+c\left(s-t\right),y\right)ds+N^{\tau}_{1}\left(a_{1},y\right)
\]
and from \eqref{eq:kqiiq-1} 
\[
\mathcal{T}^{A}_{1}\left(M\right)\left(t,x,y\right)={\displaystyle \int^{t}_{t-\frac{1}{c}x+\frac{a_{1}}{c}}}Q\left(M\right)\left(s,x+c\left(s-t\right),y\right)ds+N^{-}_{1}\left(\tau,y\right)
\]
and from \eqref{eq:kqiiq-1-1-1} 
\[
\mathcal{T}^{A}_{1}\left(M\right)\left(t,x,y\right)={\displaystyle \int^{t}_{t-\frac{1}{c}x+\frac{a_{1}}{c}}}Q\left(M\right)\left(s,x+c\left(s-t\right),y\right)ds+N^{-}_{1}\left(t-\frac{1}{c}x+\frac{a_{1}}{c},y\right)
\]
hence
\[
\mathcal{T}^{A}_{1}\left(M\right)\left(t,x,y\right)=\mathcal{T}^{B}_{1}\left(M\right)\left(t,x,y\right).
\]
 The other coincidence are verified similarly. 
\end{rem}

Let us consider the operator $\mathcal{T}:C\left(\mathscr{P},\R\right)^{4}\longrightarrow C\left(\mathscr{P},\R\right)^{4},$$M\longmapsto\mathcal{T}\left(M\right)=\left(\mathcal{T}_{i}\left(M\right)\right)^{4}_{i=1}$
where $\mathcal{T}_{i}\left(M\right),$ $i=1,2,3,4$ are given by
the formulas in \eqref{aoalal-2} to \eqref{eq:olole-1-1-1-2-2}.
It is immediate from the statements of $\left(\Sigma^{M}_{\tau,\tau'}\right)$(
\eqref{eq:ssdffz-1-1-1-1}-\eqref{eq:qqqqsa-1-1-1-1} with \eqref{eq:lsos-1}-\eqref{eq:ikioi-1-1})
and of $\Sigma_{\tau,\tau'}$ (\eqref{eq:koiqmn} with \eqref{eq:lsos-1}-\eqref{eq:ikioi-1-1})
that we have 
\begin{prop}
\label{prop:Let-the-4-tuple} Let $N=\left(N_{i}\right)^{4}_{i=1}\in C\left(\mathscr{P},\R\right)^{4}$
such that for every $i=1,2,3,4,$ $\partial_{t}N_{i},\partial_{x}N_{i}$,
$\partial_{y}N_{i}\in CB\left(\mathscr{\mathscr{O}_{\mathscr{\mathscr{\mathscr{G}\cup\mathscr{F}}}}},\R\right)$.
Then $N$ is a piecewise $C^{1}$ solution of problem $\Sigma_{\tau,\tau'}$
iff $N$ is a fixed point of $\mathcal{T}.$ 
\end{prop}

\section{Uniqueness theorem}\label{sec:Uniqueness-theorem}
\begin{prop}
For $M,N\in C\left(\begin{array}{r}
\mathscr{P}\end{array};\R\right)^{4},$it holds 
\begin{align}
\left\Vert \mathcal{T}\left(M\right)-\mathcal{T}\left(N\right)\right\Vert  & \leq\left(\tau'-\tau\right)\cdot4cS\left(\left\Vert M\right\Vert +\left\Vert N\right\Vert \right)\left\Vert M-N\right\Vert .\label{aoalal-2-1-1-1-1-1-1-1-1-2-1-1-1}
\end{align}
\end{prop}

\begin{proof}
For $M,N\in C\left(\begin{array}{r}
\mathscr{P}\end{array};\R\right)^{4},$ we have 
\begin{equation}
\left\Vert Q\left(M\right)-Q\left(N\right)\right\Vert _{\infty}\leq4cS\left(\left\Vert M\right\Vert +\left\Vert N\right\Vert \right)\left\Vert M-N\right\Vert ;\label{lopmo}
\end{equation}
this results \eqref{eq:jjdpoa} from 
\begin{multline*}
Q\left(M\right)-Q\left(N\right)=2cS\left(M_{2}-N_{2}\right)M_{3}+2cSN_{2}\left(M_{3}-N_{3}\right)\\
-2cS\left(M_{1}-N_{1}\right)M_{4}-2cSN_{1}\left(M_{4}-N_{4}\right).
\end{multline*}
In the expression \eqref{aoalal-2}-\eqref{eq:olole-1-1-2} of $\mathcal{T}_{1}\left(M\right),$
for $x-c\left(t-\tau\right)\geq a_{1},$ we have $\tau\leq t\leq\tau'$
and for $x-c\left(t-\tau\right)\leq a_{1},$ we have $\tau\leq t-\frac{1}{c}x+\frac{a_{1}}{c}\leq t\leq\tau'$.
Thus using \eqref{lopmo} we have 
\begin{align}
\left\Vert \mathcal{T}_{1}\left(M\right)-\mathcal{T}_{1}\left(N\right)\right\Vert _{\infty} & \leq\left(\tau'-\tau\right)\cdot4cS\left(\left\Vert M\right\Vert +\left\Vert N\right\Vert \right)\left\Vert M-N\right\Vert .\label{aoalal-2-1-1-1-1-1-1-1-1-2-1}
\end{align}
We have the same result for $\mathcal{T}_{i},i=2,3,4$ :

\begin{align}
\left\Vert \mathcal{T}_{i}\left(M\right)-\mathcal{T}_{i}\left(N\right)\right\Vert _{\infty} & \leq\left(\tau'-\tau\right)\cdot4cS\left(\left\Vert M\right\Vert +\left\Vert N\right\Vert \right)\left\Vert M-N\right\Vert .\label{aoalal-2-1-1-1-1-1-1-1-1-2-1-1}
\end{align}
Hence we have 
\begin{align*}
\left\Vert \mathcal{T}\left(M\right)-\mathcal{T}\left(N\right)\right\Vert  & \leq\left(\tau'-\tau\right)\cdot4cS\left(\left\Vert M\right\Vert +\left\Vert N\right\Vert \right)\left\Vert M-N\right\Vert .
\end{align*}
\end{proof}

\begin{proof}[{\large\textbf{Proof of theorem }}\textbf{\ref{thm:uniqueness}}]
Denote the restriction of a function $F$ to $\left[\eta,\eta'\right]\times\left[a_{1},b_{1}\right]\times\left[a_{2},b_{2}\right]$
by $F_{\eta,\eta'}.$\\
Suppose $M$ and $N$ are two maximal piecewise $C^{1}$ solutions
to the problem $\Sigma$ such that 

{*} the domain of $M$ (resp. $N$) is $\left[0,T_{M}\right[\times\left[a_{1},b_{1}\right]\times\left[a_{2},b_{2}\right]$
(resp. $\left[0,T_{N}\right[\times\left[a_{1},b_{1}\right]\times\left[a_{2},b_{2}\right]$);

{*} $M$ and $N$ are continuous;

{*}for every $\left[\tau,\tau'\right]\subset\left[0,T_{M}\right[$
(resp. $\left[0,T_{N}\right[$) the restriction of $M$ ( resp. $N$)
to $\left[\tau,\tau'\right]\times\left[a_{1},b_{1}\right]\times\left[a_{2},b_{2}\right]$
is in $\left(E^{\tau,\tau'}\right)^{4}.$ \\
Let $0<T<\infty$ such that $M$ and $N$ are defined on $\left[0,T\right]\times\left[a_{1},b_{1}\right]\times\left[a_{2},b_{2}\right].$
$M$ and $N$ are continuous on $\left[0,T\right]\times\left[a_{1},b_{1}\right]\times\left[a_{2},b_{2}\right]$.
Let $R\equiv\max\left\{ \left\Vert N_{0,T}\right\Vert ,\left\Vert M_{0,T}\right\Vert \right\} .$
Let $\left[\eta,\eta'\right]\subset$$\left[0,T\right].$\\
 $M_{\eta,\eta'}$ is a solution to the problem $\Sigma_{\eta,\eta'}$
where the initial data are $M_{i}\left(\eta,\cdot,\cdot\right),i=1,2,3,4.$
Since $M_{\eta,\eta'}\in\left(E^{\eta,\eta'}\right)^{4}$, $M_{i}\left(\eta,\cdot,\cdot\right)$
are (Remark \ref{rem:hsyys}) in the class of the initial data considered
in Subsection \ref{subsec:The-class-of} and there exists a finite
family $\mathscr{G}^{\eta,\eta'}_{M}$ of planes with respective equations
in the form \eqref{kdioer} such that $\partial_{t}M_{i},\partial_{x}M_{i}$
and $\partial_{y}M_{i}$ are defined on $\left]\eta,\eta'\right[\times\left]a_{1};b_{1}\right[\times\left]a_{2};b_{2}\right[$
except possibly at the points in the planes of $\mathscr{G}^{\eta,\eta'}_{M}$
; similarly, $\partial_{t}N_{i},\partial_{x}N_{i}$ and $\partial_{y}N_{i}$
are defined on $\left]\eta,\eta'\right[\times\left]a_{1};b_{1}\right[\times\left]a_{2};b_{2}\right[$
except possibly at the points in the planes of a finite family $\mathscr{G}^{\eta,\eta'}_{N}$
of planes with respective equations in the form \eqref{kdioer}. Put
$\mathscr{G}^{\eta,\eta'}=\mathscr{G}^{\eta,\eta'}_{M}\cup\mathscr{G}^{\eta,\eta'}_{N}.$
Let us denote the family $\mathscr{F}$ for a fixed $\left[\eta,\eta'\right]$
by $\mathscr{F}^{\eta,\eta'}.$ We can consider the open set $\mathscr{\mathscr{\mathscr{O}}}_{\mathscr{G}^{\eta,\eta'}\cup\mathscr{F}^{\eta,\eta'}}$
and denote by $\mathcal{T}_{\eta,\eta'}$ the operator (characterized
in Proposition \ref{prop:Let-the-4-tuple}) as following: a function
$F\in C\left(\mathscr{P},\R\right)^{4}$ such that for every $i=1,2,3,4,$
$\partial_{t}F_{i},\partial_{x}F_{i}$, $\partial_{y}F_{i}\in CB\left(\mathscr{\mathscr{\mathscr{O}}}_{\mathscr{G}^{\eta,\eta'}\cup\mathscr{F}^{\eta,\eta'}},\R\right)$
is a piecewise $C^{1}$ solution of the problem $\Sigma_{\eta,\eta'}$
iff $F$ is a fixed point of $\mathcal{T}_{\eta,\eta'}.$ Inequality
\eqref{aoalal-2-1-1-1-1-1-1-1-1-2-1-1-1} writes 
\begin{equation}
\left\Vert \mathcal{T}_{\eta,\eta'}\left(M_{\eta,\eta'}\right)-\mathcal{T}_{\eta,\eta'}\left(N_{\eta,\eta'}\right)\right\Vert \leq\left(\eta'-\eta\right)\cdot4cS\left(\left\Vert M_{\eta,\eta'}\right\Vert +\left\Vert N_{\eta,\eta'}\right\Vert \right)\left\Vert M_{\eta,\eta'}-N_{\eta,\eta'}\right\Vert .\label{aoalal-2-1-1-1-1-1-1-1-1-2-1-1-1-1}
\end{equation}
But we have $\left\Vert M_{\eta,\eta'}\right\Vert \leq\left\Vert M_{0,T}\right\Vert $
and $\left\Vert N_{\eta,\eta'}\right\Vert \leq\left\Vert N_{0,T}\right\Vert $
. Thus \ref{aoalal-2-1-1-1-1-1-1-1-1-2-1-1-1-1} implies 
\begin{align}
\left\Vert \mathcal{T}_{\eta,\eta'}\left(M_{\eta,\eta'}\right)-\mathcal{T}_{\eta,\eta'}\left(N_{\eta,\eta'}\right)\right\Vert  & \leq\left(\eta'-\eta\right)\cdot4cS\cdot2R\left\Vert M_{\eta,\eta'}-N_{\eta,\eta'}\right\Vert .\label{aoalal-2-1-1-1-1-1-1-1-1-2-1-1-1-1-1}
\end{align}
Hence if $\mathcal{T}_{\eta,\eta'}\left(M_{\eta,\eta'}\right)=M_{\eta,\eta'}$
and $\mathcal{T}_{\eta,\eta'}\left(N_{\eta,\eta'}\right)=N_{\eta,\eta'},$
we have $\left(1-\left(\eta'-\eta\right)\cdot8cSR\right)\left\Vert M_{\eta,\eta'}-N_{\eta,\eta'}\right\Vert \leq0.$

Thus for all $\left[\eta,\eta'\right]\subset\left[0,T\right]$, 
\begin{equation}
\begin{cases}
\eta'-\eta\leq\dfrac{1}{8cSR}\\
\mathcal{T}_{\eta,\eta'}\left(M_{\eta,\eta'}\right)=M_{\eta,\eta'}\text{ and }\mathcal{T}_{\eta,\eta'}\left(N_{\eta,\eta'}\right)=N_{\eta,\eta'}
\end{cases}\implies M_{\eta,\eta'}=N_{\eta,\eta'}.\label{eq:ioslq}
\end{equation}
For the problem $\Sigma$, consider the integer $n$ such that $\dfrac{n-1}{16cSR}\leq T<\dfrac{n}{16cSR}$
and $\eta_{k}=\dfrac{k}{16cSR},k\in\left\llbracket 0,n\right\rrbracket $;
consider the operators $\mathcal{T}_{\eta_{k},\eta_{k+1}},$$k\in\left\llbracket 0,n\right\rrbracket .$ 

Since $M$ and $N$ are piecewise $C^{1}$ solutions to the problem
$\Sigma$, $M_{0,\eta_{1}}$ and $N_{0,\eta_{1}}$ are piecewise $C^{1}$
solutions to the problem $\Sigma_{0,\eta_{1}}$; moreover, since $M_{0,\eta_{1}},N_{0,\eta_{1}}\in\left(E^{0,\eta_{1}}\right)^{4}$,
$\partial_{t}M_{0,\eta_{1}i},\partial_{x}M_{0,\eta_{1}i}$, $\partial_{y}M_{0,\eta_{1}i}$
as well as $\partial_{t}N_{0,\eta_{1}}{}_{i},\partial_{x}N_{0,\eta_{1}}{}_{i},\partial_{y}N_{0,\eta_{1}}{}_{i}$
are in $CB\left(\mathscr{\mathscr{\mathscr{O}}}_{\mathscr{G}^{0,\eta_{1}}\cup\mathscr{F}^{0,\eta_{1}}},\R\right)$;
hence $\mathcal{T}_{0,\eta_{1}}\left(M_{0,\eta_{1}}\right)=M_{0,\eta_{1}}\text{ and }\mathcal{T}_{0,\eta_{1}}\left(N_{0,\eta_{1}}\right)=N_{0,\eta_{1}}$;
that is $\mathcal{T}_{\eta_{0},\eta_{1}}\left(M_{\eta_{0},\eta_{1}}\right)=M_{\eta_{0},\eta_{1}}\text{ and }\mathcal{T}_{\eta_{0},\eta_{1}}\left(N_{\eta_{0},\eta_{1}}\right)=N_{\eta_{0},\eta_{1}}.$
We also have $\eta_{1}-\eta_{0}\leq\dfrac{1}{8cSR}.$ Hence from \eqref{eq:ioslq},
we deduce that $M_{\eta_{0},\eta_{1}}=N_{\eta_{0},\eta_{1}}.$

Suppose by induction that $M_{\eta_{k},\eta_{k+1}}=N_{\eta_{k},\eta_{k+1}}$
for $0\leq\eta_{k}<\eta_{k+1}<\eta_{k+2}\leq T$. Then we have $M_{\eta_{k},\eta_{k+1}}\left(\eta_{k+1},x,y\right)=N_{\eta_{k},\eta_{k+1}}\left(\eta_{k+1},x,y\right)$
for all $\left(x,y\right)\in\left[a_{1},b_{1}\right]\times\left[a_{2},b_{2}\right].$
Hence as $M$ and $N$ are piecewise $C^{1}$ solutions to the problem
$\Sigma$, $M_{\eta_{k+1},\eta_{k+2}}$ and $N_{\eta_{k+1},\eta_{k+2}}$
are piecewise $C^{1}$ solutions to the problem $\Sigma_{\eta_{k+1},\eta_{k+2}}$
with initial data $M_{\eta_{k},\eta_{k+1},i}\left(\eta_{k+1},\cdot,\cdot\right),\left(i=1,\cdots,4\right)$
where $\left(M_{\eta_{k},\eta_{k+1},i}\right)^{4}_{i=1}=M_{\eta_{k},\eta_{k+1}}.$
Since $M_{\eta_{k+1},\eta_{k+2}},N_{\eta_{k+1},\eta_{k+2}}\in\left(E^{\eta_{k+1},\eta_{k+2}}\right)^{4}$,
the derivatives of their components $M_{\eta_{k+1},\eta_{k+2}i},N_{\eta_{k+1},\eta_{k+2}i}$
are in $CB\left(\mathscr{\mathscr{\mathscr{O}}}_{\mathscr{G}^{\eta_{k+1},\eta_{k+2}}\cup\mathscr{F}^{\eta_{k+1},\eta_{k+2}}},\R\right)$;
hence we have $\mathcal{T}_{\eta_{k+1},\eta_{k+2}}\left(M_{\eta_{k+1},\eta_{k+2}}\right)=M_{\eta_{k+1},\eta_{k+2}}$
and $\mathcal{T}_{\eta_{k+1},\eta_{k+2}}\left(N_{\eta_{k+1},\eta_{k+2}}\right)=N_{\eta_{k+1},\eta_{k+2}}.$
Since $\eta_{k+2}-\eta_{k+1}\leq\dfrac{1}{8cSR}$, it follows from
\eqref{eq:ioslq} that $M_{\eta_{k+1},\eta_{k+2}}=N_{\eta_{k+1},\eta_{k+2}}.$

We deduce that $M_{\eta_{k},\eta_{k+1}}=N_{\eta_{k},\eta_{k+1}}$
for $k\in\left\llbracket 0,n-2\right\rrbracket .$ Thus $M\left(\eta_{n-1},x,y\right)=N\left(\eta_{n-1},x,y\right)$
for all $\left(x,y\right)\in\left[a_{1},b_{1}\right]\times\left[a_{2},b_{2}\right].$
Hence $M_{\eta_{n-1},T}$ and $N_{\eta_{n-1},T}$ are solutions to
the problem $\Sigma_{\eta_{n-1},T}$ with initial data $M_{\eta_{n-2},\eta_{n-1},i}\left(\eta_{n-1},\cdot,\cdot\right),\left(i=1,\cdots,4\right)$
where $\left(M_{\eta_{n-2},\eta_{n-1},i}\right)^{4}_{i=1}=M_{\eta_{n-2},\eta_{n-1}}.$
Similar arguments yield $\mathcal{T}_{\eta_{n-1},T}\left(M_{\eta_{n-1},T}\right)=M_{\eta_{n-1},T}\text{ and }\mathcal{T}_{\eta_{n-1},T}\left(N_{\eta_{n-1},T}\right)=N_{\eta_{n-1},T}.$
Then as $T-\eta_{n-1}\leq\dfrac{1}{8cSR}$ \eqref{eq:ioslq} yields
$M_{\eta_{n-1},T}=N_{\eta_{n-1},T}.$ Now we conclude that $M_{0,T}=N_{0,T}$.
And as $T$ is chosen arbitrarily in $\left[0,T_{M}\right[\cap\left[0,T_{N}\right[$
we have $M=N$ on $\left(\left[0,T_{M}\right[\cap\left[0,T_{N}\right[\right)\times\left[a_{1},b_{1}\right]\times\left[a_{2},b_{2}\right].$Assume
by contradiction that $T_{M}\neq T_{N};$ suppose $T_{M}<T_{N}.$
Then $M=N$ on $\left[0,T_{M}\right[\times\left[a_{1},b_{1}\right]\times\left[a_{2},b_{2}\right].$
Hence $N$ is an extension of $M$ to $\left[0,T_{M}\right[;$which
is a contradiction, as $M$ is a maximal solution. It follows that
$T_{M}=T_{N}$ and $M=N$.
\end{proof}

\section{Positivity of the solution}

\label{sec:Positiveness-of-the}
\begin{proof}[{\large\textbf{Proof of theorem \ref{thm:If-the-data}}}]
Let $M$ be a solution to the problem $\Sigma$ such that for every
$\left[\tau,\tau'\right]\subset\R^{+},$ the restriction of $N$ to
$\left[\tau,\tau'\right]\times\left[a_{1},b_{1}\right]\times\left[a_{2},b_{2}\right]$
is in $\left(E^{\tau,\tau'}\right)^{4}.$ Let $T>0$ such that $M$
is defined on $\left[0;T\right]$. Put $R=\left\Vert M^{0,T}\right\Vert .$
We have $R>0$ as the data are not null functions. The restriction
of $M$ to $\left[\tau,\tau'\right]\times\left[a_{1},b_{1}\right]\times\left[a_{2},b_{2}\right]$
is denoted by $M^{\tau,\tau'}.$ For any $0\leq\tau<\tau'\leq T$,
consider a finite family $\mathscr{G}^{\tau,\tau'}$ of planes with
respective equations in the form \eqref{kdioer} such that $\partial_{t}M_{i},\partial_{x}M_{i}$
and $\partial_{y}M_{i}$ are defined on $\left]\tau,\tau'\right[\times\left]a_{1};b_{1}\right[\times\left]a_{2};b_{2}\right[$
except possibly at points of the planes of $\mathscr{G}^{\tau,\tau'}$
. The family $\mathscr{F}$ for a fixed $\left[\tau,\tau'\right]$
is denoted by $\mathscr{F}^{\tau,\tau'}.$ Consider the open set $\mathscr{\mathscr{\mathscr{O}}}_{\mathscr{G}^{\tau,\tau'}\cup\mathscr{F}^{\tau,\tau'}}$
and the operator $\mathcal{T}^{\tau,\tau'}$ ( Proposition \ref{prop:Let-the-4-tuple})
such that a function $F\in C\left(\mathscr{P},\R\right)^{4}$ such
that for every $i=1,2,3,4,$ $\partial_{t}F_{i},\partial_{x}F_{i}$,
$\partial_{y}F_{i}\in CB\left(\mathscr{\mathscr{\mathscr{\mathscr{O}}}_{\mathscr{G}^{\tau,\tau'}\cup\mathscr{F}^{\tau,\tau'}}},\R\right)$
is a piecewise $C^{1}$ solution of the problem $\Sigma_{\tau,\tau'}$
iff $F$ is a fixed point of $\mathcal{T}^{\tau,\tau'}.$\\
Let $0<\tau_{1}\leq T.$ $M^{0,\tau_{1}}$ is a piecewise $C^{1}$
solution to the problem $\Sigma_{0,\tau_{1}}$ where the initial data
are $N^{0}_{i}.$ By similar arguments as in Section \ref{sec:Uniqueness-theorem},
$M^{0,\tau_{1}}$ is a fixed point of $\mathcal{T}^{0,\tau_{1}}.$
For $\left(t,x,y\right)\in\left[0,\tau_{1}\right]\times\left[a_{1},b_{1}\right]\times\left[a_{2},b_{2}\right]$,
we have (\eqref{eq:olole-1-3}) 
\begin{align}
\mathcal{T}^{0,\tau_{1},A}_{1}\left(M^{0,\tau_{1}}\right)\left(t,x,y\right) & =tQ\left(M^{0,\tau_{1}}\right)\left(s_{0},x+c\left(s_{0}-t\right),y\right)+\overline{N^{0}_{1}}\left(t,x,y\right)\label{eq:olole-1-3-1}
\end{align}
where $s_{0}\in\left[0,\tau_{1}\right]$ . We have $\left|Q\left(M^{0,\tau_{1}}\right)\right|\leq4cSR^{2}$
(\eqref{eq:jjdpoa}). Then 
\begin{equation}
\left|tQ\left(M^{0,\tau_{1}}\right)\left(s_{0},x+c\left(s_{0}-t\right),y\right)\right|\leq\tau_{1}4cSR^{2}
\end{equation}
and (\eqref{eq:kqiiq-1-1}) 
\begin{equation}
-\tau_{1}4cSR^{2}+\min_{\left[a_{1},b_{1}\right]\times\left[a_{2},b_{2}\right]}N^{0}_{1}\leq\mathcal{T}^{0,\tau_{1},A}_{1}\left(M^{0,\tau_{1}}\right)\left(t,x,y\right).
\end{equation}
Note that the bounds of each integral in the formulas \eqref{eq:olole-1-3}-\eqref{eq:olole-1-1-1-2-2}
are in the interval $\left[\tau,\tau'\right]$ . We then have (\eqref{eq:kqiiq-1-1-1})
\begin{equation}
-\tau_{1}4cSR^{2}+\min_{\left[0,T\right]\times\left[a_{2},b_{2}\right]}N^{-}_{1}\leq\mathcal{T}^{0,\tau_{1},B}_{1}\left(M^{0,\tau_{1}}\right)\left(t,x,y\right).
\end{equation}
We establish similar relations for $\mathcal{T}^{0,\tau_{1},A}_{i}$
and $\mathcal{T}^{0,\tau_{1},B}_{i}$ ($i=2,3,4$). Then we see that,
if 
\begin{multline}
T<\frac{1}{4cSR^{2}}\min_{1\leq i\leq4}\left\{ \right.\min_{\left[a_{1},b_{1}\right]\times\left[a_{2},b_{2}\right]}N^{0}_{i},\min_{\left[0,T\right]\times\left[a_{2},b_{2}\right]}N^{-}_{1},\\
\min_{\left[0,T\right]\times\left[a_{1},b_{1}\right]}N^{-}_{2},\min_{\left[0,T\right]\times\left[a_{1},b_{1}\right]}N^{+}_{3},\min_{\left[0,T\right]\times\left[a_{2},b_{2}\right]}N^{+}_{4}\left.\right\} \label{mpod}
\end{multline}
then $\mathcal{T}^{0,T}\left(M^{0,T}\right)=M^{0,T}>0.$ If \eqref{mpod}
does not hold, then taking 
\begin{multline}
\tau_{1}=\frac{1}{8cSR^{2}}\min_{1\leq i\leq4}\left\{ \right.\min_{\left[a_{1},b_{1}\right]\times\left[a_{2},b_{2}\right]}N^{0}_{i},\min_{\left[0,T\right]\times\left[a_{2},b_{2}\right]}N^{-}_{1},\min_{\left[0,T\right]\times\left[a_{1},b_{1}\right]}N^{-}_{2},\\
\min_{\left[0,T\right]\times\left[a_{1},b_{1}\right]}N^{+}_{3},\min_{\left[0,T\right]\times\left[a_{2},b_{2}\right]}N^{+}_{4}\left.\right\} \label{mpod-1}
\end{multline}
yields $\mathcal{T}^{0,\tau_{1}}\left(M^{0,\tau_{1}}\right)=M^{0,\tau_{1}}>0.$

Let $\tau_{2}$ such that $\tau_{1}<\tau_{2}\leq T.$ Consider the
problem $\Sigma_{\tau_{1},\tau_{2}}$ where the initial data are $N^{\tau_{1}}\equiv M^{0,\tau_{1}}\left(\tau_{1},\cdot,\cdot\right).$
$M^{\tau_{1},\tau_{2}}$ is a fixed point of $\mathcal{T}^{\tau_{1},\tau_{2}}.$
We have as previously, for $i=1,2,3,4$ 
\begin{equation}
-\left(\tau_{2}-\tau_{1}\right)4cSR^{2}+\min_{\left[a_{1},b_{1}\right]\times\left[a_{2},b_{2}\right]}N^{\tau_{1}}_{i}\leq\mathcal{T}^{\tau_{1},\tau_{2},A}_{i}\left(M^{\tau_{1},\tau_{2}}\right)\left(t,x,y\right)
\end{equation}
and similar relations for $\mathcal{T}^{\tau_{1},\tau_{2},B}_{i}.$
Hence if 
\begin{multline}
T-\tau_{1}<\frac{1}{4cSR^{2}}\min_{1\leq i\leq4}\left\{ \right.\min_{\left[a_{1},b_{1}\right]\times\left[a_{2},b_{2}\right]}N^{\tau_{1}}_{i},\min_{\left[0,T\right]\times\left[a_{2},b_{2}\right]}N^{-}_{1},\\
\min_{\left[0,T\right]\times\left[a_{1},b_{1}\right]}N^{-}_{2},\min_{\left[0,T\right]\times\left[a_{1},b_{1}\right]}N^{+}_{3},\min_{\left[0,T\right]\times\left[a_{2},b_{2}\right]}N^{+}_{4}\left.\right\} \label{mpod-2}
\end{multline}
then $M^{\tau_{1},T}>0$. If \eqref{mpod-2} does not hold, then for
$\tau_{2}$such that 
\begin{multline}
\tau_{2}-\tau_{1}=\frac{1}{8cSR^{2}}\min_{1\leq i\leq4}\left\{ \right.\min_{\left[a_{1},b_{1}\right]\times\left[a_{2},b_{2}\right]}N^{\tau_{1}}_{i},\min_{\left[0,T\right]\times\left[a_{2},b_{2}\right]}N^{-}_{1},\min_{\left[0,T\right]\times\left[a_{1},b_{1}\right]}N^{-}_{2},\\
\min_{\left[0,T\right]\times\left[a_{1},b_{1}\right]}N^{+}_{3},\min_{\left[0,T\right]\times\left[a_{2},b_{2}\right]}N^{+}_{4}\left.\right\} \label{mpod-2-1}
\end{multline}
we have $M^{\tau_{1},\tau_{2}}>0$.

Suppose by induction that we have $0<\tau_{1}<\cdots<\tau_{n}\leq T$
such that $M^{\tau_{k-1},\tau_{k}}>0$ for $k=2,\cdots,n.$ By similar
arguments as previously, take $N^{\tau_{n}}\equiv M^{\tau_{n-1},\tau_{n}}\left(\tau_{n},\cdot,\cdot\right).$
If 
\begin{multline}
T-\tau_{n}<\frac{1}{4cSR^{2}}\min_{1\leq i\leq4}\left\{ \right.\min_{\left[a_{1},b_{1}\right]\times\left[a_{2},b_{2}\right]}N^{\tau_{n}}_{i},\min_{\left[0,T\right]\times\left[a_{2},b_{2}\right]}N^{-}_{1},\\
\min_{\left[0,T\right]\times\left[a_{1},b_{1}\right]}N^{-}_{2},\min_{\left[0,T\right]\times\left[a_{1},b_{1}\right]}N^{+}_{3},\min_{\left[0,T\right]\times\left[a_{2},b_{2}\right]}N^{+}_{4}\left.\right\} \label{mpod-2-2}
\end{multline}
then $M^{\tau_{n},T}>0$. If \eqref{mpod-2-2} does not hold, then
for $\tau_{n+1}$ such that 
\begin{multline}
\tau_{n+1}-\tau_{n}=\frac{1}{8cSR^{2}}\min_{1\leq i\leq4}\left\{ \right.\min_{\left[a_{1},b_{1}\right]\times\left[a_{2},b_{2}\right]}N^{\tau_{n}}_{i},\min_{\left[0,T\right]\times\left[a_{2},b_{2}\right]}N^{-}_{1},\\
\min_{\left[0,T\right]\times\left[a_{1},b_{1}\right]}N^{-}_{2},\min_{\left[0,T\right]\times\left[a_{1},b_{1}\right]}N^{+}_{3},\min_{\left[0,T\right]\times\left[a_{2},b_{2}\right]}N^{+}_{4}\left.\right\} \label{mpod-2-1-1}
\end{multline}
we have $M^{\tau_{n},\tau_{n+1}}>0$.

Now suppose that the positivity breaks down for an interval $\left[0,T'\right[$
such that $T'<T$. Then $M^{0,T'}>0$ on $\left[0,T'\right[$ and
by continuity of $M,$ $M^{0,T'}>0$ on $\left[0,T'\right].$ Then
if 
\begin{multline}
T-T'<\frac{1}{4cSR^{2}}\min_{1\leq i\leq4}\left\{ \right.\min_{\left[a_{1},b_{1}\right]\times\left[a_{2},b_{2}\right]}M^{0,T'}_{i}\left(T',\cdot,\cdot\right),\min_{\left[0,T\right]\times\left[a_{2},b_{2}\right]}N^{-}_{1},\\
\min_{\left[0,T\right]\times\left[a_{1},b_{1}\right]}N^{-}_{2},\min_{\left[0,T\right]\times\left[a_{1},b_{1}\right]}N^{+}_{3},\min_{\left[0,T\right]\times\left[a_{2},b_{2}\right]}N^{+}_{4}\left.\right\} \label{mpod-2-2-1}
\end{multline}
then we obtain $M^{T',T}>0;$otherwise taking $T''\in\left]T',T\right]$
such that 
\begin{multline}
T''-T'=\frac{1}{8cSR^{2}}\min_{1\leq i\leq4}\left\{ \right.\min_{\left[a_{1},b_{1}\right]\times\left[a_{2},b_{2}\right]}M^{0,T'}_{i}\left(T',\cdot,\cdot\right),\min_{\left[0,T\right]\times\left[a_{2},b_{2}\right]}N^{-}_{1},\\
\min_{\left[0,T\right]\times\left[a_{1},b_{1}\right]}N^{-}_{2},\min_{\left[0,T\right]\times\left[a_{1},b_{1}\right]}N^{+}_{3},\min_{\left[0,T\right]\times\left[a_{2},b_{2}\right]}N^{+}_{4}\left.\right\} \label{mpod-2-1-1-1}
\end{multline}
yields $M^{T',T''}>0.$ Thus, the positivity is satisfied on $\left[0,T\right].$
Hence we have $M^{0,T}>0,$ for all $T$ such that $M$ is defined
on $\left[0,T\right];$ we conclude that $M>0.$ 
\end{proof}

\section{Existence theorem for the bounded time interval}\label{sec:Existence-theorem-for}

In this section, $\left[\tau,\tau'\right]$ is fixed and we consider
the open set $\mathscr{O}_{\mathscr{F}}$ defined in \eqref{kssos}
and the operator $\mathcal{T}:C\left(\mathscr{P},\R\right)^{4}\longrightarrow C\left(\mathscr{P},\R\right)^{4}$
which satisfies the following property: a function $N=\left(N_{i}\right)^{4}_{i=1}\in C\left(\mathscr{P},\R\right)^{4}$
such that for every $i=1,2,3,4,$ $\partial_{t}N_{i},\partial_{x}N_{i}$,
$\partial_{y}N_{i}\in CB\left(\mathscr{\mathscr{O}_{\mathscr{\mathscr{\mathscr{F}}}}},\R\right)$,
is a piecewise $C^{1}$ solution to the problem $\Sigma_{\tau,\tau'}$
iff $N$ is a fixed point of $\mathcal{T}$ (Proposition \ref{prop:Let-the-4-tuple}).
\\
We have $\mathcal{T}:C\left(\mathscr{P},\R\right)^{4}\longrightarrow C\left(\mathscr{P},\R\right)^{4},$$M\longmapsto\mathcal{T}\left(M\right)=\left(\mathcal{T}_{i}\left(M\right)\right)^{4}_{i=1}$
where $\mathcal{T}_{i}\left(M\right),$ $i=1,2,3,4$ are given by
the formulas in \eqref{aoalal-2} to \eqref{eq:olole-1-1-1-2-2}.
\\
$\left(C\left(\mathscr{P},\R\right)^{4},\left\Vert \cdot\right\Vert \right)$
is a Banach space where $\left\Vert M\right\Vert =\max_{1\leq i\leq4}\|M_{i}\|_{\infty}$
for every $M=\left(M_{i}\right)^{4}_{i=1}\in C\left(\mathscr{P},\R\right)^{4}.$
Let $R>0$. Let us consider the closed ball $\mathcal{B}_{R}$ in
$\left(C\left(\mathscr{P},\R\right),\left\Vert \cdot\right\Vert \right)$
defined by 
\begin{equation}
\mathcal{B}_{R}=\left\{ M\in C\left(\mathscr{P},\R\right)^{4}:\left\Vert M\right\Vert \leq R\right\} .\label{ldodd}
\end{equation}

\begin{prop}
For every $M\in\mathcal{B}_{R}$, 
\begin{align}
\left\Vert \mathcal{T}\left(M\right)\right\Vert  & \leq\left(\tau'-\tau\right)4cSR^{2}+q.\label{eq:ksoloa-2-1-1-1}
\end{align}
\end{prop}

\begin{proof}
It follows from \eqref{eq:jjdpoa} that for every $M\in C\left(\mathscr{P},\R\right)^{4},$
\begin{equation}
\left\Vert Q\left(M\right)\right\Vert _{\infty}\leq4cS\left\Vert M\right\Vert ^{2},\label{kdiod}
\end{equation}
 Using \eqref{kdiod} and \eqref{bghghf}, \eqref{eq:olole-1-3} yields 

\begin{align}
\left\Vert \mathcal{T}^{A}_{1}\left(M\right)\right\Vert _{\infty} & \leq\left(\tau'-\tau\right)4cS\left\Vert M\right\Vert ^{2}+q\label{eq:ksoloa-2}
\end{align}
Similarly from \eqref{eq:olole-1-1-2}, it follows 
\begin{equation}
\left\Vert \mathcal{T}^{B,\sigma}_{1}\left(M\right)\right\Vert _{\infty}\leq\left(\tau'-\tau\right)4cS\left\Vert M\right\Vert ^{2}+q\label{eq:omplo-1}
\end{equation}
hence from \eqref{aoalal-2}, 
\begin{align}
\left\Vert \mathcal{T}_{1}\left(M\right)\right\Vert _{\infty} & \leq\left(\tau'-\tau\right)4cS\left\Vert M\right\Vert ^{2}+q.\label{eq:ksoloa-2-1}
\end{align}
Similarly, \eqref{eq:olole-1-1-2}-\eqref{eq:olole-1-1-1-2-2} yield
the same bound for $\left\Vert \mathcal{T}_{i}\left(M\right)\right\Vert _{\infty}$
when $i=2,3,4$. It follows that 
\begin{align}
\left\Vert \mathcal{T}\left(M\right)\right\Vert  & \leq\left(\tau'-\tau\right)4cS\left\Vert M\right\Vert ^{2}+q.\label{eq:ksoloa-2-1-1}
\end{align}
\end{proof}

\begin{prop}
\label{prop:Assume-that--1}Assume that $\tau'-\tau<\dfrac{1}{16cSq}.$
Then the operator \textup{$\mathcal{T}$ }admits an unique fixed point
$N^{\tau,\tau'}=\left(N^{\tau,\tau'}_{i}\right)^{4}_{i=1}$ such that
\begin{equation}
\left\Vert N^{\tau,\tau'}\right\Vert \leq\dfrac{1-\sqrt{1-16cS\left(\tau'-\tau\right)q}}{8cS\left(\tau'-\tau\right)}\label{skeut-2}
\end{equation}
 where $q$ is defined in \eqref{bghghf}. 
\end{prop}

\begin{proof}
Let $R>0$. It follows from \eqref{aoalal-2-1-1-1-1-1-1-1-1-2-1-1-1}
that for all $M,N\in\mathcal{B}_{R},$
\begin{align}
\left\Vert \mathcal{T}\left(M\right)-\mathcal{T}\left(N\right)\right\Vert  & \leq8cS\left(\tau'-\tau\right)R\left\Vert M-N\right\Vert .\label{aoalal-2-1-1-1-1-1-1-1-1-2-1-1-1-2}
\end{align}
We have 
\begin{equation}
\left(\tau'-\tau\right)4cSR^{2}+q\leq R\iff\begin{cases}
1-16cS\left(\tau'-\tau\right)q\geq0\\
\dfrac{1-\sqrt{1-16cS\left(\tau'-\tau\right)q}}{8cS\left(\tau'-\tau\right)}\leq R\leq\dfrac{1-\sqrt{1-16cS\left(\tau'-\tau\right)q}}{8cS\left(\tau'-\tau\right)}
\end{cases}.\label{ujddu}
\end{equation}
 Take $\tau'-\tau<\dfrac{1}{16cSq}$ and 
\begin{equation}
R=R_{\tau,\tau',q}\equiv\dfrac{1-\sqrt{1-16cS\left(\tau'-\tau\right)q}}{8cS\left(\tau'-\tau\right)}.\label{ujizk}
\end{equation}
Then from \eqref{eq:ksoloa-2-1-1-1} and \eqref{ujddu}, we have $\left\Vert \mathcal{T}\left(M\right)\right\Vert \leq R_{\tau,\tau',q}$
for every $M\in\mathcal{B}_{R_{\tau,\tau',q}},$ that is 
\begin{equation}
\mathcal{T}\left(\mathcal{B}_{R_{\tau,\tau',q}}\right)\subset\mathcal{B}_{R_{\tau,\tau',q}}.\label{odldoe}
\end{equation}
It also holds that 
\begin{equation}
8cS\left(\tau'-\tau\right)R_{\tau,\tau',q}=1-\sqrt{1-16cS\left(\tau'-\tau\right)q}<1.\label{kdiid}
\end{equation}
 Hence from \eqref{aoalal-2-1-1-1-1-1-1-1-1-2-1-1-1-2}, it follows
that $\mathcal{T}$ defines a contraction on the complete space $\left(\mathcal{B}_{R_{\tau,\tau',q}},\left\Vert \cdot\right\Vert \right)$
. Hence by the Banach fixed-point theorem, $\mathcal{T}$ admits an
unique fixed point $N^{\tau,\tau'}$ in $\mathcal{B}_{R_{\tau,\tau',q}}.$
\end{proof}

\begin{prop}
\label{prop:Assume-that-}Assume that $\tau'-\tau<\dfrac{1}{16cSq}.$
Then the functions $\overline{N^{\tau}_{1}}$, $\overline{N^{-}_{1}}$,
$\overline{N^{\tau}_{2}}$, $\overline{N^{-}_{2}}$, $\overline{N^{\tau}_{3}}$
, $\overline{N^{+}_{3}}$, $\overline{N^{\tau}_{4}}$ and $\overline{N^{+}_{4}}$
(\eqref{eq:kqiiq-1-1} - \eqref{eq:looqp-1-1-1-1}) satisfy:\textup{
\begin{equation}
\Bigl\Vert\overline{N^{\tau}_{i}}\Bigr\Vert_{\infty},\Bigl\Vert\overline{N^{-}_{j}}\Bigr\Vert_{\infty},\Bigl\Vert\overline{N^{+}_{k}}\Bigr\Vert_{\infty}\leq R_{\tau,\tau',q}\label{slsosz}
\end{equation}
for $i=1,2,3,4$ ; $j=1,2$; $k=3,4$.}
\end{prop}

\begin{proof}
We have $\Bigl\Vert\overline{N^{\tau}_{i}}\Bigr\Vert_{\infty}\leq\left\Vert N^{\tau}_{i}\right\Vert _{\infty}$
and $\Bigl\Vert\overline{N^{\pm}_{i}}\Bigr\Vert_{\infty}\leq\left\Vert N^{\pm}_{i}\right\Vert _{\infty}$
hence 
\[
\Bigl\Vert\overline{N^{\tau}_{i}}\Bigr\Vert_{\infty},\Bigl\Vert\overline{N^{\pm}_{i}}\Bigr\Vert_{\infty}\leq\max_{1\leq i\leq4}\Biggl\{\left\Vert N^{\tau}_{i}\right\Vert _{\infty},\left\Vert N^{-}_{1}\right\Vert _{\infty},\left\Vert N^{-}_{2}\right\Vert _{\infty},\left\Vert N^{+}_{3}\right\Vert _{\infty},\left\Vert N^{+}_{4}\right\Vert _{\infty}\Biggr\}=q.
\]
We have 
\begin{align*}
\dfrac{q}{R_{\tau,\tau',q}} & =\dfrac{8cS\left(\tau'-\tau\right)q}{1-\sqrt{1-16cS\left(\tau'-\tau\right)q}}\\
 & =\dfrac{8cS\left(\tau'-\tau\right)q\left(1+\sqrt{1-16cS\left(\tau'-\tau\right)q}\right)}{16cS\left(\tau'-\tau\right)q}\\
 & =\dfrac{1+\sqrt{1-16cS\left(\tau'-\tau\right)q}}{2}<1
\end{align*}
hence $\Bigl\Vert\overline{N^{\tau}_{i}}\Bigr\Vert_{\infty},\Bigl\Vert\overline{N^{\pm}_{i}}\Bigr\Vert_{\infty}\leq R_{\tau,\tau',q}.$
\end{proof}

Let us consider the following sequence $\left(M_{\left(n\right)}\right)_{n\geq0}$
of functions in $C\left(\mathscr{P},\R\right)^{4}$ defined by:
\begin{equation}
\forall n\geq0,\,M_{\left(n\right)}=\left(M_{\left(n\right)i}\right)^{4}_{i=1}\label{jjsiz}
\end{equation}
\begin{equation}
M_{\left(0\right)1}=\overline{N^{\tau}_{1}}\cdot\mathbb{I}_{x-c\left(t-\tau\right)\geq a_{1}}+\overline{N^{-}_{1}}\cdot\mathbb{I}_{x-c\left(t-\tau\right)\leq a_{1}}\label{aoalal-2-1}
\end{equation}
\begin{equation}
M_{\left(0\right)2}=\overline{N^{\tau}_{2}}\cdot\mathbb{I}_{y-c\left(t-\tau\right)\geq a_{2}}+\overline{N^{-}_{2}}\cdot\mathbb{I}_{y-c\left(t-\tau\right)\leq a_{2}}\label{aoalal-1-2-1}
\end{equation}

\begin{equation}
M_{\left(0\right)3}=\overline{N^{\tau}_{3}}\cdot\mathbb{I}_{y+c\left(t-\tau\right)\leq b_{2}}+\overline{N^{+}_{3}}\cdot\mathbb{I}_{y+c\left(t-\tau\right)\geq b_{2}}\label{aoalal-1-1-3-1}
\end{equation}

\begin{equation}
M_{\left(0\right)4}=\overline{N^{\tau}_{4}}\cdot\mathbb{I}_{x+c\left(t-\tau\right)\leq b_{1}}+\overline{N^{+}_{4}}\cdot\mathbb{I}_{x+c\left(t-\tau\right)\geq b_{1}}\label{aoalal-1-1-1-1-1}
\end{equation}
and
\begin{equation}
\forall n\geq0,\,M_{\left(n+1\right)}=\mathcal{T}\left(M_{\left(n\right)}\right).\label{jdufe}
\end{equation}

Note that for every $i=1,2,3,4,$ $M_{\left(0\right)i}$ is continuous
on $\mathscr{P}$ as a consequence of the compatibility conditions
\eqref{eq:kqiiq}-\eqref{eq:looqp-1}. 
\begin{prop}
\label{prop:kiidid}Assume that $\tau'-\tau<\dfrac{1}{16cSq}.$ Then
for every $n\geq0$, $M_{\left(n\right)}$ is in the closed ball $\mathcal{B}_{R_{\tau,\tau',q}}$
in \textup{$\left(C\left(\mathscr{P},\R\right),\left\Vert \cdot\right\Vert \right)$. }
\end{prop}

\begin{proof}
From Proposition \ref{prop:Assume-that-} and \eqref{aoalal-2-1}-\eqref{aoalal-1-1-1-1-1},
it follows that $M_{\left(0\right)}$ is in the closed ball $\mathcal{B}_{R_{\tau,\tau',q}}.$
From \eqref{odldoe}, it holds that $M_{\left(n\right)}\in\mathcal{B}_{R_{\tau,\tau',q}}$
implies that $M_{\left(n+1\right)}=\mathcal{T}\left(M_{\left(n\right)}\right)\in\mathcal{B}_{R_{\tau,\tau',q}}.$ 
\end{proof}

\begin{prop}
Assume that $\tau'-\tau<\dfrac{1}{16cSq}.$ For every $i=1,2,3,4,$
the sequence $\left(M_{\left(n\right)i}\right)_{n\geq0}$(\eqref{jjsiz})
converges uniformly on the compact $\mathscr{P}$ to the function
$N^{\tau,\tau'}_{i}$ where $N^{\tau,\tau'}=\left(N^{\tau,\tau'}_{i}\right)^{4}_{i=1}$
is the unique fixed point of \textup{$\mathcal{T}$ in $\mathcal{B}_{R_{\tau,\tau',q}}.$}
\end{prop}

\begin{proof}
From the Banach fixed-point theorem, the iterated $\left(M_{\left(n\right)}\right)_{n\geq0}$
(\eqref{jjsiz}-\eqref{jdufe}) in the space $\left(\mathcal{B}_{R_{\tau,\tau',q}},\left\Vert \cdot\right\Vert \right)$
converges to the fixed point $N^{\tau,\tau'}=\left(N^{\tau,\tau'}_{i}\right)^{4}_{i=1}$
. Since $\left\Vert M_{\left(n\right)}-N^{\tau,\tau'}\right\Vert =\max_{1\leq i\leq4}\|M_{\left(n\right)i}-N^{\tau,\tau'}_{i}\|_{\infty}$,
it follows that for every $i=1,2,3,4,$ $\|M_{\left(n\right)i}-N^{\tau,\tau'}_{i}\|_{\infty}=\sup_{\mathscr{P}}\left|M_{\left(n\right)i}-N^{\tau,\tau'}_{i}\right|$
converges to $0$. 
\end{proof}

\begin{prop}
\label{rem:From--we}For every $i=1,2,3,4$, $\partial_{t}M_{\left(n\right)i},\partial_{x}M_{\left(n\right)i}$,
$\partial_{y}M_{\left(n\right)i}\in CB\left(\mathscr{\mathscr{O}_{\mathscr{\mathscr{\mathscr{F}}}}},\R\right).$
\end{prop}

\begin{proof}
From Remark \ref{rem:sjsusuus}, the first order derivatives of $\overline{N^{\tau}_{1}}$,
$\overline{N^{-}_{1}}$, $\overline{N^{\tau}_{2}}$, $\overline{N^{-}_{2}}$,
$\overline{N^{\tau}_{3}}$ , $\overline{N^{+}_{3}}$, $\overline{N^{\tau}_{4}}$
and $\overline{N^{+}_{4}}$ (\eqref{eq:kqiiq-1-1} - \eqref{eq:looqp-1-1-1-1})
are in $CB\left(\mathscr{\mathscr{O}_{\mathscr{\mathscr{\mathscr{F}}}}},\R\right).$
Hence $\partial_{t}M_{\left(0\right)i},\partial_{x}M_{\left(0\right)i}$,
$\partial_{y}M_{\left(0\right)i}\in CB\left(\mathscr{\mathscr{O}_{\mathscr{\mathscr{\mathscr{F}}}}},\R\right).$
Now, from the discussion in Remark \ref{rem:ssskzkz}, if $M=\left(M_{i}\right)^{4}_{i=1}\in C\left(\mathscr{P},\R\right)^{4}$
is such that for every $i=1,2,3,4,$  $\partial_{t}M_{i},\partial_{x}M_{i}$,
$\partial_{y}M_{i}\in CB\left(\mathscr{O}_{\mathscr{\mathscr{\mathscr{G}\cup\mathscr{F}}}},\R\right)$,
then for every $i=1,2,3,4,$$\mathcal{T}\left(M\right)\in C\left(\mathscr{P},\R\right)^{4}$and
$\partial_{t}\mathcal{T}_{i}\left(M\right),$$\partial_{x}\mathcal{T}_{i}\left(M\right),$$\partial_{y}\mathcal{T}_{i}\left(M\right)\in CB\left(\mathscr{O}_{\mathscr{\mathscr{\mathscr{G}\cup\mathscr{F}}}},\R\right)$.
It follows from \eqref{jdufe}, that for every $i=1,2,3,4$, $\partial_{t}M_{\left(n\right)i},\partial_{x}M_{\left(n\right)i}$,
$\partial_{y}M_{\left(n\right)i}\in CB\left(\mathscr{\mathscr{O}_{\mathscr{\mathscr{\mathscr{F}}}}},\R\right).$
\end{proof}

\begin{prop}
Let $M=\left(M_{i}\right)^{4}_{i=1}\in C\left(\mathscr{P},\R\right)^{4}$
such that for every $i=1,2,3,4,$ $\partial_{t}M_{i},\partial_{x}M_{i}$,
$\partial_{y}M_{i}\in CB\left(\mathscr{O}_{\mathscr{F}},\R\right).$
Then 
\begin{multline}
\max_{1\leq i\leq4}\left\Vert \nabla\mathcal{T}_{i}\left(M\right)\right\Vert \leq\left(1+\frac{1}{c}\right)8cS\left\Vert M\right\Vert ^{2}+\\
\left(\tau'-\tau\right)\left(1+c\right)\cdot8cS\max_{1\leq i\leq4}\left\Vert \nabla M_{i}\right\Vert \left\Vert M\right\Vert +\left(1+c+\dfrac{1}{c}\right)q'\label{eq:ksoloa-1-2-4-1-1}
\end{multline}
$q'$ defined in \eqref{oloo-1}. 
\end{prop}

\begin{proof}
Let $M=\left(M_{i}\right)^{4}_{i=1}\in C\left(\mathscr{P},\R\right)^{4}$
such that for every $i=1,2,3,4,$ $\partial_{t}M_{i},\partial_{x}M_{i}$,
$\partial_{y}M_{i}\in CB\left(\mathscr{O}_{\mathscr{F}},\R\right)$.
\eqref{eq:olole-1-3} and \eqref{eq:olole-1-1-2} yield 
\begin{multline}
\dfrac{\partial\mathcal{T}^{A}_{1}\left(M\right)}{\partial t}\left(t,x,y\right)=Q\left(M\right)\left(t,x,y\right)+\\
\int^{t}_{\tau}-c\dfrac{\partial Q\left(M\right)}{\partial x}\left(s,x+c\left(s-t\right),y\right)ds+\dfrac{\partial\overline{N^{\tau}_{1}}}{\partial t}\left(t,x,y\right)\label{looen}
\end{multline}
\begin{align}
\dfrac{\partial\mathcal{T}^{A}_{1}\left(M\right)}{\partial x}\left(t,x,y\right) & ={\displaystyle \int^{t}_{\tau}}\dfrac{\partial Q\left(M\right)}{\partial x}\left(s,x+c\left(s-t\right),y\right)ds+\dfrac{\partial\overline{N^{\tau}_{1}}}{\partial x}\left(t,x,y\right)\label{olqooq}
\end{align}
\begin{align}
\dfrac{\partial\mathcal{T}^{A}_{1}\left(M\right)}{\partial y}\left(t,x,y\right) & ={\displaystyle \int^{t}_{\tau}}\dfrac{\partial Q\left(M\right)}{\partial y}\left(s,x+c\left(s-t\right),y\right)ds+\dfrac{\partial\overline{N^{\tau}_{1}}}{\partial y}\left(t,x,y\right).\label{osllso}
\end{align}
\begin{align}
\dfrac{\partial\mathcal{T}^{B}_{1}\left(M\right)}{\partial t}\left(t,x,y\right) & =Q\left(M\right)\left(t,x,y\right){\displaystyle -Q\left(M\right)\left(t-\frac{1}{c}x+\frac{a_{1}}{c},a_{1},y\right)+}\label{osloz}\\
 & {\displaystyle \int^{t}_{t-\frac{1}{c}x+\frac{a_{1}}{c}}}-c\dfrac{\partial Q\left(M\right)}{\partial x}\left(s,x+c\left(s-t\right),y\right)ds+\dfrac{\partial\overline{N^{-}_{1}}}{\partial t}\left(t,x,y\right)\nonumber 
\end{align}
\begin{align}
\dfrac{\partial\mathcal{T}^{B}_{1}\left(M\right)}{\partial x}\left(t,x,y\right) & =\frac{1}{c}Q\left(M\right)\left(t-\frac{1}{c}x+\frac{a_{1}}{c},a_{1},y\right)+\nonumber \\
 & {\displaystyle \int^{t}_{t-\frac{1}{c}x+\frac{a_{1}}{c}}}\dfrac{\partial Q\left(M\right)}{\partial x}\left(s,x+c\left(s-t\right),y\right)ds+\dfrac{\partial\overline{N^{-}_{1}}}{\partial x}\left(t,x,y\right)\label{oolzoo}
\end{align}
 and 
\begin{align}
\dfrac{\partial\mathcal{T}^{B}_{1}\left(M\right)}{\partial y}\left(t,x,y\right) & ={\displaystyle \int^{t}_{t-\frac{1}{c}x+\frac{a_{1}}{c}}}\dfrac{\partial Q\left(M\right)}{\partial y}\left(s,x+c\left(s-t\right),y\right)ds+\dfrac{\partial\overline{N^{-}_{1}}}{\partial y}\left(t,x,y\right).\label{olqqoa}
\end{align}
 From \eqref{eq:jjdpoa} it follows $\left\Vert Q\left(M\right)\right\Vert _{\infty}\leq4cS\left\Vert M\right\Vert ^{2}$
. We have 

${\displaystyle \dfrac{\partial Q\left(M\right)}{\partial x}=2cS\left(\dfrac{\partial M_{2}}{\partial x}M_{3}+M_{2}\dfrac{\partial M_{3}}{\partial x}-\dfrac{\partial M_{1}}{\partial x}M_{4}-M_{1}\dfrac{\partial M_{4}}{\partial x}\right)}$
and a similar expression for

$\dfrac{\partial Q\left(M\right)}{\partial y}$ . Hence we have 
\begin{equation}
\left\Vert \dfrac{\partial Q\left(M\right)}{\partial x}\right\Vert _{\infty},\left\Vert \dfrac{\partial Q\left(M\right)}{\partial y}\right\Vert _{\infty}\leq8cS\max_{1\leq i\leq4}\left\Vert \nabla M_{i}\right\Vert \left\Vert M\right\Vert .\label{eq:lopoi-1}
\end{equation}
It follows from \eqref{looen}-\eqref{olqqoa} that 
\begin{equation}
\left\Vert \dfrac{\partial\mathcal{T}^{A}_{1}\left(M\right)}{\partial t}\right\Vert _{\infty}\leq4cS\left\Vert M\right\Vert ^{2}+\left(\tau'-\tau\right)c\cdot8cS\max_{1\leq i\leq4}\left\Vert \nabla M_{i}\right\Vert \left\Vert M\right\Vert +\left\Vert \nabla\overline{N^{\tau}_{1}}\right\Vert \label{eq:lopiok-1-1}
\end{equation}
\begin{equation}
\left\Vert \dfrac{\partial\mathcal{T}^{A}_{1}\left(M\right)}{\partial x}\right\Vert _{\infty},\left\Vert \dfrac{\partial\mathcal{T}^{A}_{1}\left(M\right)}{\partial y}\right\Vert _{\infty}\leq\left(\tau'-\tau\right)8cS\max_{1\leq i\leq4}\left\Vert \nabla M_{i}\right\Vert \left\Vert M\right\Vert +\left\Vert \nabla\overline{N^{\tau}_{1}}\right\Vert \label{eq:lopiok-1-1-1}
\end{equation}
and 
\begin{equation}
\left\Vert \dfrac{\partial\mathcal{T}^{B}_{1}\left(M\right)}{\partial t}\right\Vert _{\infty}\leq8cS\left\Vert M\right\Vert ^{2}+\left(\tau'-\tau\right)c\cdot8cS\max_{1\leq i\leq4}\left\Vert \nabla M_{i}\right\Vert \left\Vert M\right\Vert +\left\Vert \nabla\overline{N^{-}_{1}}\right\Vert \label{eq:lopiok-1-1-3}
\end{equation}
\begin{equation}
\left\Vert \dfrac{\partial\mathcal{T}^{B}_{1}\left(M\right)}{\partial x}\right\Vert _{\infty}\leq\frac{1}{c}8cS\left\Vert M\right\Vert ^{2}+\left(\tau'-\tau\right)8cS\max_{1\leq i\leq4}\left\Vert \nabla M_{i}\right\Vert \left\Vert M\right\Vert +\left\Vert \nabla\overline{N^{-}_{1}}\right\Vert \label{eq:lopiok-1-1-1-2}
\end{equation}
\begin{equation}
\left\Vert \dfrac{\partial\mathcal{T}^{B}_{1}\left(M\right)}{\partial y}\right\Vert _{\infty}\leq\left(\tau'-\tau\right)8cS\max_{1\leq i\leq4}\left\Vert \nabla M_{i}\right\Vert \left\Vert M\right\Vert +\left\Vert \nabla\overline{N^{-}_{1}}\right\Vert .\label{eq:lopiok-1-1-1-2-1}
\end{equation}
\eqref{eq:lopiok-1-1}-\eqref{eq:lopiok-1-1-1-2-1} yield 
\begin{multline}
\left\Vert \nabla\mathcal{T}_{1}\left(M\right)\right\Vert \leq\left(1+\frac{1}{c}\right)8cS\left\Vert M\right\Vert ^{2}+\\
\left(\tau'-\tau\right)\left(1+c\right)\cdot8cS\max_{1\leq i\leq4}\left\Vert \nabla M_{i}\right\Vert \left\Vert M\right\Vert +\max\left\{ \left\Vert \nabla\overline{N^{\tau}_{1}}\right\Vert ,\left\Vert \nabla\overline{N^{-}_{1}}\right\Vert \right\} .\label{eq:ksoloa-1-2}
\end{multline}
Similarly \eqref{aoalal-1-2}-\eqref{eq:olole-1-1-1-2-2} yield the
following estimations 
\begin{multline}
\left\Vert \nabla\mathcal{T}_{2}\left(M\right)\right\Vert \leq\left(1+\frac{1}{c}\right)8cS\left\Vert M\right\Vert ^{2}+\\
\left(\tau'-\tau\right)\left(1+c\right)\cdot8cS\max_{1\leq i\leq4}\left\Vert \nabla M_{i}\right\Vert \left\Vert M\right\Vert +\max\left\{ \left\Vert \nabla\overline{N^{\tau}_{2}}\right\Vert ,\left\Vert \nabla\overline{N^{-}_{2}}\right\Vert \right\} \label{eq:ksoloa-1-2-1}
\end{multline}
\begin{multline}
\left\Vert \nabla\mathcal{T}_{3}\left(M\right)\right\Vert \leq\left(1+\frac{1}{c}\right)8cS\left\Vert M\right\Vert ^{2}+\\
\left(\tau'-\tau\right)\left(1+c\right)\cdot8cS\max_{1\leq i\leq4}\left\Vert \nabla M_{i}\right\Vert \left\Vert M\right\Vert +\max\left\{ \left\Vert \nabla\overline{N^{\tau}_{3}}\right\Vert ,\left\Vert \nabla\overline{N^{+}_{3}}\right\Vert \right\} \label{eq:ksoloa-1-2-2}
\end{multline}
\begin{multline}
\left\Vert \nabla\mathcal{T}_{4}\left(M\right)\right\Vert \leq\left(1+\frac{1}{c}\right)8cS\left\Vert M\right\Vert ^{2}+\\
\left(\tau'-\tau\right)\left(1+c\right)\cdot8cS\max_{1\leq i\leq4}\left\Vert \nabla M_{i}\right\Vert \left\Vert M\right\Vert +\max\left\{ \left\Vert \nabla\overline{N^{\tau}_{4}}\right\Vert ,\left\Vert \nabla\overline{N^{+}_{4}}\right\Vert \right\} .\label{eq:ksoloa-1-2-3}
\end{multline}
It follows that 
\begin{multline}
\max_{1\leq i\leq4}\left\Vert \nabla\mathcal{T}_{i}\left(M\right)\right\Vert \leq\left(1+\frac{1}{c}\right)8cS\left\Vert M\right\Vert ^{2}+\\
\left(\tau'-\tau\right)\left(1+c\right)\cdot8cS\max_{1\leq i\leq4}\left\Vert \nabla M_{i}\right\Vert \left\Vert M\right\Vert +\max_{1\leq i\leq4}\left\{ \left\Vert \nabla\overline{N^{\tau}_{i}}\right\Vert ,\left\Vert \nabla\overline{N^{\pm}_{i}}\right\Vert \right\} .\label{eq:ksoloa-1-2-4}
\end{multline}
We have $\dfrac{\partial\overline{N^{\tau}_{1}}}{\partial t}\left(t,x,y\right)=-c\dfrac{\partial N^{\tau}_{1}}{\partial x}\left(x-c\left(t-\tau\right),y\right),$$\dfrac{\partial\overline{N^{\tau}_{1}}}{\partial x}\left(t,x,y\right)=\dfrac{\partial N^{\tau}_{1}}{\partial x}\left(x-c\left(t-\tau\right),y\right),$
$\dfrac{\partial\overline{N^{\tau}_{1}}}{\partial y}\left(t,x,y\right)=\dfrac{\partial N^{\tau}_{1}}{\partial y}\left(x-c\left(t-\tau\right),y\right).$
Hence $\left\Vert \nabla\overline{N^{\tau}_{1}}\right\Vert \leq\left(1+c\right)\left\Vert \nabla N^{\tau}_{1}\right\Vert .$
We verify from \eqref{eq:lqoop-1-1}-\eqref{eq:looqp-1-1-1} and \eqref{eq:kqiiq-1-1-1}-\eqref{eq:looqp-1-1-1-1}
that for $i=1,2,3,4,$ $\left\Vert \nabla\overline{N^{\tau}_{i}}\right\Vert \leq\left(1+c\right)\left\Vert \nabla N^{\tau}_{i}\right\Vert ,$
$\left\Vert \nabla\overline{N^{+}_{i}}\right\Vert \leq\left(1+\dfrac{1}{c}\right)\left\Vert \nabla N^{+}_{i}\right\Vert ,\left(i=3,4\right),$
$\left\Vert \nabla\overline{N^{-}_{i}}\right\Vert \leq\left(1+\dfrac{1}{c}\right)\left\Vert \nabla N^{-}_{i}\right\Vert ,\left(i=1,2\right).$
Hence 
\begin{equation}
\max_{1\leq i\leq4}\left\{ \left\Vert \nabla\overline{N^{\tau}_{i}}\right\Vert ,\left\Vert \nabla\overline{N^{-}_{1}}\right\Vert ,\left\Vert \nabla\overline{N^{-}_{2}}\right\Vert ,\left\Vert \nabla\overline{N^{+}_{3}}\right\Vert ,\left\Vert \nabla\overline{N^{+}_{4}}\right\Vert \right\} \leq\left(1+c+\dfrac{1}{c}\right)q'\label{ksiso}
\end{equation}
since ( \eqref{oloo-1}) $q'=\max_{1\leq i\leq4}\left\{ \left\Vert \nabla N^{\tau}_{i}\right\Vert ,\left\Vert \nabla N^{-}_{1}\right\Vert ,\left\Vert \nabla N^{-}_{2}\right\Vert ,\left\Vert \nabla N^{+}_{3}\right\Vert ,\left\Vert \nabla N^{+}_{4}\right\Vert \right\} $
. \eqref{eq:ksoloa-1-2-4} implies that 
\begin{multline*}
\max_{1\leq i\leq4}\left\Vert \nabla\mathcal{T}_{i}\left(M\right)\right\Vert \leq\left(1+\frac{1}{c}\right)8cS\left\Vert M\right\Vert ^{2}+\\
\left(\tau'-\tau\right)\left(1+c\right)\cdot8cS\max_{1\leq i\leq4}\left\Vert \nabla M_{i}\right\Vert \left\Vert M\right\Vert +\left(1+c+\dfrac{1}{c}\right)q'.
\end{multline*}
\end{proof}

\begin{prop}
\label{prop:kkxiii}Assume that $\tau'-\tau<\frac{1+2c}{\left(1+c\right)^{2}}\cdot\dfrac{1}{16cSq}<\dfrac{1}{16cSq}.$
Then $1-8cS\left(1+c\right)\left(\tau'-\tau\right)R_{\tau,\tau',q}>0$.
For every $i=1,2,3,4,$ the sequence $\left(\left\Vert \nabla M_{\left(n\right)i}\right\Vert \right)_{n\geq0}$
(\eqref{iidod}) is bounded : for every $n\geq0,$
\begin{equation}
\left\Vert \nabla M_{\left(n\right)i}\right\Vert \leq\dfrac{\left(1+\frac{1}{c}\right)8cSR^{2}_{\tau,\tau',q}+\left(1+c+\frac{1}{c}\right)q'}{1-8cS\left(1+c\right)\left(\tau'-\tau\right)R_{\tau,\tau',q}}.\label{susjsu}
\end{equation}
\end{prop}

\begin{proof}
We have the following equivalences 
\[
\tau'-\tau<\dfrac{1}{\left(1+c\right)8cSR_{\tau,\tau',q}}\iff\tau'-\tau<\dfrac{1}{\left(1+c\right)8cS}\cdot\dfrac{8cS\left(\tau'-\tau\right)}{1-\sqrt{1-16cS\left(\tau'-\tau\right)q}}
\]
 {}
\[
\iff1<\dfrac{1}{\left(1+c\right)}\cdot\dfrac{1}{1-\sqrt{1-16cS\left(\tau'-\tau\right)q}}
\]
\[
\iff1-\sqrt{1-16cS\left(\tau'-\tau\right)q}<\dfrac{1}{1+c}
\]
\[
\iff\dfrac{c}{1+c}<\sqrt{1-16cS\left(\tau'-\tau\right)q}\iff16cS\left(\tau'-\tau\right)q<\frac{1+2c}{\left(1+c\right)^{2}}.
\]
 Hence $1-8cS\left(1+c\right)\left(\tau'-\tau\right)R_{\tau,\tau',q}>0$
holds. Put 
\begin{equation}
R'\equiv\dfrac{\left(1+\frac{1}{c}\right)8cSR^{2}_{\tau,\tau',q}+\left(1+c+\frac{1}{c}\right)q'}{1-8cS\left(1+c\right)\left(\tau'-\tau\right)R_{\tau,\tau',q}}.\label{lddoo}
\end{equation}
From \eqref{aoalal-2-1}-\eqref{aoalal-1-1-1-1-1}, we have for every
$i=1,2,3,4,$
\begin{equation}
\left\Vert \nabla M_{\left(0\right)i}\right\Vert \leq\max_{1\leq i\leq4}\left\{ \left\Vert \nabla\overline{N^{\tau}_{i}}\right\Vert ,\left\Vert \nabla\overline{N^{-}_{1}}\right\Vert ,\left\Vert \nabla\overline{N^{-}_{2}}\right\Vert ,\left\Vert \nabla\overline{N^{+}_{3}}\right\Vert ,\left\Vert \nabla\overline{N^{+}_{4}}\right\Vert \right\} .\label{lsooso}
\end{equation}
Hence (\eqref{ksiso}), for every $i=1,2,3,4,$
\begin{equation}
\left\Vert \nabla M_{\left(0\right)i}\right\Vert \leq\left(1+c+\frac{1}{c}\right)q'.\label{lsooso-1}
\end{equation}
Since $0<1-8cS\left(1+c\right)\left(\tau'-\tau\right)R_{\tau,\tau',q}<1$
(\eqref{kdiid}), we have 
\begin{equation}
\left(1+c+\frac{1}{c}\right)q'\leq\dfrac{\left(1+c+\frac{1}{c}\right)q'}{1-8cS\left(1+c\right)\left(\tau'-\tau\right)R_{\tau,\tau',q}}\leq R',\label{jsusuz}
\end{equation}
 hence, for every $i=1,2,3,4,$
\begin{equation}
\left\Vert \nabla M_{\left(0\right)i}\right\Vert \leq R'.\label{ikxkxi}
\end{equation}
Let $M=\left(M_{i}\right)^{4}_{i=1}\in C\left(\mathscr{P},\R\right)^{4}$
such that for every $i=1,2,3,4,$ $\partial_{t}M_{i},\partial_{x}M_{i}$,
$\partial_{y}M_{i}\in CB\left(\mathscr{O}_{\mathscr{F}},\R\right)$.
Suppose that $\left\Vert M\right\Vert \leq R_{\tau,\tau',q}$ and
$\max_{1\leq i\leq4}\left\Vert \nabla M_{i}\right\Vert \leq R'.$
\eqref{eq:ksoloa-1-2-4-1-1} yields 
\begin{multline}
\max_{1\leq i\leq4}\left\Vert \nabla\mathcal{T}_{i}\left(M\right)\right\Vert \leq\left(1+\frac{1}{c}\right)8cSR^{2}_{\tau,\tau',q}+\\
\left(\tau'-\tau\right)\left(1+c\right)\cdot8cSR'R_{\tau,\tau',q}+\left(1+c+\dfrac{1}{c}\right)q'\label{eq:ksoloa-1-2-4-1-1-1}
\end{multline}
From \eqref{lddoo}, 
\begin{equation}
\left(1+\frac{1}{c}\right)8cSR^{2}_{\tau,\tau',q}+\left(1+c+\dfrac{1}{c}\right)q'+\left(\tau'-\tau\right)\left(1+c\right)\cdot8cSR'R_{\tau,\tau',q}=R'\label{ksiis}
\end{equation}
hence 
\begin{equation}
\max_{1\leq i\leq4}\left\Vert \nabla\mathcal{T}_{i}\left(M\right)\right\Vert \leq R'.\label{kisis}
\end{equation}
Since for every $n\geq0,$$M_{\left(n\right)}$ is in the closed ball
$\mathcal{B}_{R_{\tau,\tau',q}}$ (Propo. \ref{prop:kiidid}) and
$\forall n\geq0,\,M_{\left(n+1\right)}=\mathcal{T}\left(M_{\left(n\right)}\right)$
(\eqref{jdufe}), and for every $i=1,2,3,4,$$\left\Vert \nabla M_{\left(0\right)i}\right\Vert \leq R'$,
we conclude by induction that for every $i=1,2,3,4,$an every $n\geq0,$
\begin{equation}
\left\Vert \nabla M_{\left(n\right)i}\right\Vert \leq R'.\label{ikxkxi-2}
\end{equation}
\end{proof}

Let $M=\left(M_{i}\right)^{4}_{i=1}\in C\left(\mathscr{P},\R\right)^{4}$
such that for every $i=1,2,3,4,$ $\partial_{t}M_{i},\partial_{x}M_{i}$,
$\partial_{y}M_{i}\in CB\left(\mathscr{\mathscr{O}_{\mathscr{\mathscr{\mathscr{F}}}}},\R\right).$
For every $\partial_{\alpha}=\partial_{t},\partial_{x},\partial_{y}$
, put 
\begin{equation}
\partial_{\alpha}M\equiv\left(\partial_{\alpha}M_{i}\right)^{4}_{i=1}\in CB\left(\mathscr{\mathscr{O}_{\mathscr{\mathscr{\mathscr{F}}}}},\R\right)^{4}.\label{osoos}
\end{equation}

\begin{lem}
\label{lem:looszee}Assume that $\tau'-\tau<\frac{1+2c}{\left(1+c\right)^{2}}\cdot\dfrac{1}{16cSq}.$
Let $M,N\in\mathcal{B}_{R_{\tau,\tau',q}}$ such that for every $i=1,2,3,4,$
$\partial_{t}M_{i},\partial_{x}M_{i}$, $\partial_{y}M_{i},$ $\partial_{t}N_{i},\partial_{x}N_{i}$,
$\partial_{y}N_{i}\in CB\left(\mathscr{\mathscr{O}_{\mathscr{\mathscr{\mathscr{F}}}}},\R\right)$,
and $\left\Vert \nabla M_{i}\right\Vert ,$ $\left\Vert \nabla N_{i}\right\Vert \leq R'$,
$R'$ defined in \eqref{lddoo}. Put 
\begin{equation}
K\equiv8cS\left(\tau'-\tau\right)R_{\tau,\tau',q}\label{lod}
\end{equation}
and 
\begin{equation}
A\equiv\left(2+\frac{1}{c}\right)8cSR_{\tau,\tau',q}+\left(\tau'-\tau\right)\left(1+c\right)8cSR'\label{dkid}
\end{equation}
 Then we have for every $\partial_{\alpha}=\partial_{t},\partial_{x},\partial_{y},$
\begin{equation}
\left\Vert \partial_{\alpha}\mathcal{T}\left(M\right)-\partial_{\alpha}\mathcal{T}\left(N\right)\right\Vert \leq K\left\Vert \partial_{\alpha}M-\partial_{\alpha}N\right\Vert +A\left\Vert M-N\right\Vert .\label{ldodi}
\end{equation}
\end{lem}

\begin{proof}
Let $\partial_{\alpha}=\partial_{t},\partial_{x},\partial_{y}.$ We
have
\begin{equation}
\dfrac{\partial Q\left(M\right)}{\partial\alpha}=2cS\left(\dfrac{\partial M_{2}}{\partial\alpha}M_{3}+M_{2}\dfrac{\partial M_{3}}{\partial\alpha}-\dfrac{\partial M_{1}}{\partial\alpha}M_{4}-M_{1}\dfrac{\partial M_{4}}{\partial\alpha}\right).\label{ldod}
\end{equation}
For $i,j=1,2,3,4,$
\[
\dfrac{\partial M_{i}}{\partial\alpha}M_{j}-\dfrac{\partial N_{i}}{\partial\alpha}N_{j}=\left(\partial_{\alpha}M_{i}-\partial_{\alpha}N_{i}\right)M_{j}+\dfrac{\partial N_{i}}{\partial\alpha}\left(M_{j}-N_{j}\right)
\]
hence 
\begin{equation}
\left\Vert \dfrac{\partial M_{i}}{\partial\alpha}M_{j}-\dfrac{\partial N_{i}}{\partial\alpha}N_{j}\right\Vert _{\infty}\leq R_{\tau,\tau',q}\left\Vert \partial_{\alpha}M-\partial_{\alpha}N\right\Vert +R'\left\Vert M-N\right\Vert .\label{ksiis-1}
\end{equation}
 \eqref{ldod} and \eqref{ksiis-1} yield 
\begin{equation}
\left\Vert \partial_{\alpha}Q\left(M\right)-\partial_{\alpha}Q\left(N\right)\right\Vert _{\infty}\leq8cSR_{\tau,\tau',q}\left\Vert \partial_{\alpha}M-\partial_{\alpha}N\right\Vert +8cSR'\left\Vert M-N\right\Vert .\label{lcoco}
\end{equation}
From \eqref{lopmo}, we have 
\begin{equation}
\left\Vert Q\left(M\right)-Q\left(N\right)\right\Vert _{\infty}\leq8cSR_{\tau,\tau',q}\left\Vert M-N\right\Vert .\label{lcoco-1}
\end{equation}
Using \eqref{lcoco} and \eqref{lcoco-1}, \eqref{looen} and \eqref{osloz}
yield 
\begin{multline}
\left\Vert \partial_{t}\mathcal{T}_{1}\left(M\right)-\partial_{t}\mathcal{T}_{1}\left(N\right)\right\Vert _{\infty}\leq16cSR_{\tau,\tau',q}\left\Vert M-N\right\Vert +8cS\left(\tau'-\tau\right)cR_{\tau,\tau',q}\left\Vert \partial_{t}M-\partial_{t}N\right\Vert +\\
8cS\left(\tau'-\tau\right)cR'\left\Vert M-N\right\Vert .\label{jsuss}
\end{multline}
Similarly, from \eqref{eq:olole-1-1-1-1-1-1-1}-\eqref{eq:olole-1-1-1-2-2},
we obtain that \eqref{jsuss} holds for $\left\Vert \partial_{t}\mathcal{T}_{i}\left(M\right)-\partial_{t}\mathcal{T}_{i}\left(N\right)\right\Vert _{\infty}$,
$i=2,3,4$. We deduce that 
\begin{multline}
\left\Vert \partial_{t}\mathcal{T}\left(M\right)-\partial_{t}\mathcal{T}\left(N\right)\right\Vert _{\infty}\leq8cS\left(\tau'-\tau\right)cR_{\tau,\tau',q}\left\Vert \partial_{t}M-\partial_{t}N\right\Vert +\\
\left[16cSR_{\tau,\tau',q}+8cS\left(\tau'-\tau\right)cR'\right]\left\Vert M-N\right\Vert .\label{jsuss-1}
\end{multline}
Using similar calculations and estimates, we obtain for $i=1,2,3,4,$
\begin{multline}
\left\Vert \partial_{x}\mathcal{T}_{i}\left(M\right)-\partial_{x}\mathcal{T}_{i}\left(N\right)\right\Vert _{\infty}\leq\frac{1}{c}8cSR_{\tau,\tau',q}\left\Vert M-N\right\Vert +8cS\left(\tau'-\tau\right)cR\left\Vert \partial_{x}M-\partial_{x}N\right\Vert +\\
8cS\left(\tau'-\tau\right)R'\left\Vert M-N\right\Vert ,\label{jsuss-2}
\end{multline}
hence 
\begin{multline}
\left\Vert \partial_{x}\mathcal{T}\left(M\right)-\partial_{x}\mathcal{T}\left(N\right)\right\Vert _{\infty}\leq8cS\left(\tau'-\tau\right)R_{\tau,\tau',q}\left\Vert \partial_{x}M-\partial_{x}N\right\Vert +\\
\left[\frac{1}{c}8cSR_{\tau,\tau',q}+8cS\left(\tau'-\tau\right)R'\right]\left\Vert M-N\right\Vert .\label{jsuss-1-1}
\end{multline}
Similarly, 
\begin{multline}
\left\Vert \partial_{y}\mathcal{T}\left(M\right)-\partial_{y}\mathcal{T}\left(N\right)\right\Vert _{\infty}\leq8cS\left(\tau'-\tau\right)R_{\tau,\tau',q}\left\Vert \partial_{y}M-\partial_{y}N\right\Vert +\\
\left[\frac{1}{c}8cSR_{\tau,\tau',q}+8cS\left(\tau'-\tau\right)R'\right]\left\Vert M-N\right\Vert .\label{jsuss-1-1-1}
\end{multline}
From \eqref{jsuss-1}, \eqref{jsuss-1-1} and \eqref{jsuss-1-1-1},
we obtain that for $\partial_{\alpha}=\partial_{t},\partial_{x},\partial_{y}$
\begin{multline}
\left\Vert \partial_{\alpha}\mathcal{T}\left(M\right)-\partial_{\alpha}\mathcal{T}\left(N\right)\right\Vert _{\infty}\leq8cS\left(\tau'-\tau\right)cR_{\tau,\tau',q}\left\Vert \partial_{\alpha}M-\partial_{\alpha}N\right\Vert +\\
\left[\left(2+\frac{1}{c}\right)8cSR_{\tau,\tau',q}+\left(\tau'-\tau\right)\left(1+c\right)8cSR'\right]\left\Vert M-N\right\Vert .\label{jsuss-1-2}
\end{multline}
\end{proof}

\begin{prop}
Assume that $\tau'-\tau<\frac{1+2c}{\left(1+c\right)^{2}}\cdot\dfrac{1}{16cSq}.$
Then for every $i=1,2,3,4,$ the sequences $\left(\partial_{t}M_{\left(n\right)i}\right)_{n\geq0},\left(\partial_{x}M_{\left(n\right)i}\right)_{n\geq0}$
and $\left(\partial_{y}M_{\left(n\right)i}\right)_{n\geq0}$ in $CB\left(\mathscr{\mathscr{O}_{\mathscr{\mathscr{\mathscr{F}}}}},\R\right)$
are uniformly convergent on $\mathscr{O}_{\mathscr{\mathscr{\mathscr{F}}}}.$
\end{prop}

\begin{proof}
Any couple of terms of the sequence $\left(M_{\left(n\right)}\right)_{n\geq0}$
satisfy the assumptions in Lemma. \ref{lem:looszee}. Using $M_{\left(n+1\right)}=\mathcal{T}\left(M_{\left(n\right)}\right),$
$n\geq0$ $(\ref{jdufe})$ and the estimate \eqref{ldodi}, we have
\begin{align}
\left\Vert \partial_{\alpha}M_{\left(n+1\right)}-\partial_{\alpha}M_{\left(n\right)}\right\Vert  & \leq K\left\Vert \partial_{\alpha}M_{\left(n\right)}-\partial_{\alpha}M_{\left(n-1\right)}\right\Vert +A\left\Vert M_{\left(n\right)}-M_{\left(n-1\right)}\right\Vert \label{sksie}\\
K\left\Vert \partial_{\alpha}M_{\left(n\right)}-\partial_{\alpha}M_{\left(n-1\right)}\right\Vert  & \leq K^{2}\left\Vert \partial_{\alpha}M_{\left(n-1\right)}-\partial_{\alpha}M_{\left(n-2\right)}\right\Vert +AK\left\Vert M_{\left(n-1\right)}-M_{\left(n-2\right)}\right\Vert \nonumber \\
K^{2}\left\Vert \partial_{\alpha}M_{\left(n-1\right)}-\partial_{\alpha}M_{\left(n-2\right)}\right\Vert  & \leq K^{3}\left\Vert \partial_{\alpha}M_{\left(n-2\right)}-\partial_{\alpha}M_{\left(n-3\right)}\right\Vert +AK^{2}\left\Vert M_{\left(n-2\right)}-M_{\left(n-3\right)}\right\Vert \nonumber \\
............. & ...............\nonumber \\
K^{n-2}\left\Vert \partial_{\alpha}M_{\left(3\right)}-\partial_{\alpha}M_{\left(2\right)}\right\Vert  & \leq K^{n-1}\left\Vert \partial_{\alpha}M_{\left(2\right)}-\partial_{\alpha}M_{\left(1\right)}\right\Vert +AK^{n-2}\left\Vert M_{\left(2\right)}-M_{\left(1\right)}\right\Vert \nonumber \\
K^{n-1}\left\Vert \partial_{\alpha}M_{\left(2\right)}-\partial_{\alpha}M_{\left(1\right)}\right\Vert  & \leq K^{n}\left\Vert \partial_{\alpha}M_{\left(1\right)}-\partial_{\alpha}M_{\left(0\right)}\right\Vert +AK^{n-1}\left\Vert M_{\left(1\right)}-M_{\left(0\right)}\right\Vert .\label{lsoos}
\end{align}
We have $\left\Vert M_{\left(n+1\right)}-M_{\left(n\right)}\right\Vert \leq K\left\Vert M_{\left(n\right)}-M_{\left(n-1\right)}\right\Vert $
since $\mathcal{T}$ is a $K=8cS\left(\tau'-\tau\right)R_{\tau,\tau',q}$
contraction on $\mathcal{B}_{R_{\tau,\tau',q}}.$ Hence 
\begin{align}
\left\Vert M_{\left(n\right)}-M_{\left(n-1\right)}\right\Vert  & \leq K^{n-1}\left\Vert M_{\left(1\right)}-M_{\left(0\right)}\right\Vert \label{lsoso}\\
K\left\Vert M_{\left(n-1\right)}-M_{\left(n-2\right)}\right\Vert  & \leq K\cdot K^{n-2}\left\Vert M_{\left(1\right)}-M_{\left(0\right)}\right\Vert \nonumber \\
K^{2}\left\Vert M_{\left(n-2\right)}-M_{\left(n-3\right)}\right\Vert  & \leq K^{2}\cdot K^{n-3}\left\Vert M_{\left(1\right)}-M_{\left(0\right)}\right\Vert \nonumber \\
...... & .......\nonumber \\
K^{n-2}\left\Vert M_{\left(2\right)}-M_{\left(1\right)}\right\Vert  & \leq K^{n-2}\cdot K\left\Vert M_{\left(1\right)}-M_{\left(0\right)}\right\Vert \label{ksisi}
\end{align}
Hence from \eqref{sksie}-\eqref{lsoos} and \eqref{lsoso}-\eqref{ksisi},
we have 
\begin{equation}
\left\Vert \partial_{\alpha}M_{\left(n+1\right)}-\partial_{\alpha}M_{\left(n\right)}\right\Vert \leq K^{n}\left\Vert \partial_{\alpha}M_{\left(1\right)}-\partial_{\alpha}M_{\left(0\right)}\right\Vert +A\cdot nK^{n-1}\left\Vert M_{\left(1\right)}-M_{\left(0\right)}\right\Vert .\label{hdydy}
\end{equation}
Therefore for every $p\geq1,$
\begin{equation}
\left\Vert \partial_{\alpha}M_{\left(n+p\right)}-\partial_{\alpha}M_{\left(n\right)}\right\Vert \leq\sum^{n+p-1}_{j=n}K^{j}\left\Vert \partial_{\alpha}M_{\left(1\right)}-\partial_{\alpha}M_{\left(0\right)}\right\Vert +A\sum^{n+p-1}_{j=n}jK^{j-1}\left\Vert M_{\left(1\right)}-M_{\left(0\right)}\right\Vert \label{kkssi}
\end{equation}
\begin{multline}
\left\Vert \partial_{\alpha}M_{\left(n+p\right)}-\partial_{\alpha}M_{\left(n\right)}\right\Vert \leq K^{n}\frac{1-K^{p}}{1-K}\left\Vert \partial_{\alpha}M_{\left(1\right)}-\partial_{\alpha}M_{\left(0\right)}\right\Vert +\\
A\left(n-p+1\right)K^{n}\frac{1-K^{p}}{1-K}K^{-1}\left\Vert M_{\left(1\right)}-M_{\left(0\right)}\right\Vert .\label{jssu}
\end{multline}
Since $0<K<1,$ $\lim_{n\rightarrow\infty}K^{n}=0=\lim_{n\rightarrow\infty}\left(n-p+1\right)K^{n}$,
hence $\lim_{n\rightarrow\infty}\left\Vert \partial_{\alpha}M_{\left(n+p\right)}-\partial_{\alpha}M_{\left(n\right)}\right\Vert =0;$
it follows that $\lim_{n\rightarrow\infty}\left\Vert \partial_{\alpha}M_{\left(n+p\right)i}-\partial_{\alpha}M_{\left(n\right)i}\right\Vert _{\infty}=0;$
since $\left(CB\left(\mathscr{\mathscr{O}_{\mathscr{\mathscr{\mathscr{F}}}}},\R\right),\left\Vert \right\Vert _{\infty}\right)$
is complete, $\left(\partial_{\alpha}M_{\left(n\right)i}\right)_{n\geq0}$
is uniformly convergent on $\mathscr{\mathscr{O}_{\mathscr{\mathscr{\mathscr{F}}}}}.$ 
\end{proof}

\begin{proof}[{\large\textbf{Proof of Theorem }}{\large\ref{thm:Suppose-.-Then}}]
Since $\tau'-\tau<\frac{1+2c}{\left(1+c\right)^{2}}\cdot\dfrac{1}{16cSq}$
then from Proposition \ref{prop:Assume-that--1}, the operator $\mathcal{T}$
admits an unique continuous fixed point $N^{\tau,\tau'}=\left(N^{\tau,\tau'}_{i}\right)^{4}_{i=1}$
such that $\left\Vert N^{\tau,\tau'}\right\Vert \leq R_{\tau,\tau',q}$.
The inequality $\tau'-\tau<\frac{1}{8cS\left(1+c\right)R_{\tau,\tau',q}}$
is obtained in Proposition \ref{prop:kkxiii}.\\
We have constructed for every $i=1,2,3,4,$ the sequence $\left(M_{\left(n\right)i}\right)_{n\geq0}$
which converges uniformly on $\mathscr{P}$ to the component $N^{\tau,\tau'}_{i}$
of the fixed point $N^{\tau,\tau'}$ of the operator $\mathcal{T}$
and we have proved that the sequence $\left(\partial_{\alpha}M_{\left(n\right)i}\right)_{n\geq0}$
converges uniformly on $\mathscr{\mathscr{O}_{\mathscr{\mathscr{\mathscr{F}}}}}\subset\mathscr{P}$
. It follows that the derivative $\partial_{\alpha}N^{\tau,\tau'}_{i}$
exists on $\mathscr{\mathscr{O}_{\mathscr{\mathscr{\mathscr{F}}}}}$and
that $\partial_{\alpha}N^{\tau,\tau'}_{i}$ is the uniform limit of
$\left(\partial_{\alpha}M_{\left(n\right)i}\right)_{n\geq0}$ on $\mathscr{\mathscr{O}_{\mathscr{\mathscr{\mathscr{F}}}}}.$
Since $\left(\partial_{\alpha}M_{\left(n\right)i}\right)_{n\geq0}\subset\left(CB\left(\mathscr{\mathscr{O}_{\mathscr{\mathscr{\mathscr{F}}}}},\R\right),\left\Vert \right\Vert _{\infty}\right),$
and $\left\Vert \partial_{\alpha}M_{\left(n\right)i}\right\Vert _{\infty}\leq R',$
it results that $\partial_{\alpha}N^{\tau,\tau'}_{i}\in CB\left(\mathscr{\mathscr{O}_{\mathscr{\mathscr{\mathscr{F}}}}},\R\right)$
and $\left\Vert \partial_{\alpha}N^{\tau,\tau'}_{i}\right\Vert _{\infty}\leq R'$
where $R'=\dfrac{\left(1+\frac{1}{c}\right)8cSR^{2}_{\tau,\tau',q}+\left(1+c+\frac{1}{c}\right)q'}{1-8cS\left(1+c\right)\left(\tau'-\tau\right)R_{\tau,\tau',q}}.$
\\
Now from Proposition \ref{prop:Let-the-4-tuple}, $N^{\tau,\tau'}$
is a piecewise $C^{1}$ solution to the problem $\Sigma_{\tau,\tau'}$.
From Definition \ref{def:jsusooa} and Subsection \ref{subsec:The-class-of-1},
$N^{\tau,\tau'}_{i}\in C\left(\mathscr{P},\R\right)$ and $\partial_{\alpha}N^{\tau,\tau'}_{i}\in CB\left(\mathscr{\mathscr{O}_{\mathscr{\mathscr{\mathscr{F}}}}},\R\right)$
implies that $N^{\tau,\tau'}=\left(N^{\tau,\tau'}_{i}\right)^{4}_{i=1}\in\left(E^{\tau,\tau'}\right)^{4}$.
\\
Since the data are assumed to be positive, Theorem \ref{thm:If-the-data}
implies that $N^{\tau,\tau'}_{i}$is positive.
\end{proof}

\begin{proof}[{\large\textbf{Proof of Corollary }}{\large\ref{cor:lsoosz}}]
$\tau'-\tau=\min\left\{ \frac{1}{8cSR_{0}+\frac{\left(1+c\right)^{2}}{1+2c}16cSq},\frac{R_{0}-q}{8cSR^{2}_{0}}\right\} $
implies that $\tau'-\tau<\frac{1+2c}{\left(1+c\right)^{2}}\cdot\dfrac{1}{16cSq}.$
According to Theorem \ref{thm:Suppose-.-Then}, the problem $\Sigma_{\tau,\tau'}$
admits an unique positive piecewise $C^{1}$ solution $N^{\tau,\tau'}=\left(N^{\tau,\tau'}_{i}\right)^{4}_{i=1}\in\left(E^{\tau,\tau'}\right)^{4}$
such that $\left\Vert N^{\tau,\tau'}\right\Vert \leq\dfrac{1-\sqrt{1-16cS\left(\tau'-\tau\right)q}}{8cS\left(\tau'-\tau\right)}=R_{\tau,\tau',q}.$
\\
Take $R_{0}$ such that $R_{0}>q.$ 
\[
R_{\tau,\tau',q}<R_{0}\iff\dfrac{1-\sqrt{1-16cS\left(\tau'-\tau\right)q}}{8cS\left(\tau'-\tau\right)}<R_{0}
\]
\[
\iff1-8cS\left(\tau'-\tau\right)R_{0}<\sqrt{1-16cS\left(\tau'-\tau\right)q}.
\]
$\tau'-\tau=\min\left\{ \frac{1}{8cSR_{0}+\frac{\left(1+c\right)^{2}}{1+2c}16cSq},\frac{R_{0}-q}{8cSR^{2}_{0}}\right\} $
also implies that $\tau'-\tau<\dfrac{1}{8cSR_{0}}$. Hence 
\begin{multline*}
1-8cS\left(\tau'-\tau\right)R_{0}<\sqrt{1-16cS\left(\tau'-\tau\right)q}\iff\\
1-16cS\left(\tau'-\tau\right)R_{0}+\left(8cS\right)^{2}\left(\tau'-\tau\right)^{2}R^{2}_{0}<1-16cS\left(\tau'-\tau\right)q
\end{multline*}
\[
\iff\left(8cS\right)^{2}\left(\tau'-\tau\right)^{2}R^{2}_{0}<16cS\left(\tau'-\tau\right)\left(R_{0}-q\right)
\]
\[
\iff\tau'-\tau<\dfrac{R_{0}-q}{4cSR_{0}}.
\]
But, $\tau'-\tau=\min\left\{ \frac{1}{8cSR_{0}+\frac{\left(1+c\right)^{2}}{1+2c}16cSq},\frac{R_{0}-q}{8cSR^{2}_{0}}\right\} $
implies also that $\tau'-\tau\leq\dfrac{R_{0}-q}{8cSR_{0}}<\dfrac{R_{0}-q}{4cSR_{0}}.$
Hence we have 
\begin{equation}
\left\Vert N^{\tau,\tau'}\right\Vert \leq R_{\tau,\tau',q}<R_{0}\label{jsuus}
\end{equation}
\end{proof}

\begin{rem}
\label{rem:jsuuaaa}Let us put 
\begin{equation}
q_{f}\equiv\max_{1\leq i\leq4}\Biggl\{\left\Vert N^{\tau,\tau'}_{i}\left(\tau',\cdot,\cdot\right)\right\Vert _{\infty},\left\Vert N^{-}_{1}\right\Vert _{\infty},\left\Vert N^{-}_{2}\right\Vert _{\infty},\left\Vert N^{+}_{3}\right\Vert _{\infty},\left\Vert N^{+}_{4}\right\Vert _{\infty}\Biggr\}.\label{jssuus}
\end{equation}
We have ( \eqref{bghghf}) , $q=\max_{1\leq i\leq4}\Biggl\{\left\Vert N^{\tau}_{i}\right\Vert _{\infty},\left\Vert N^{-}_{1}\right\Vert _{\infty},\left\Vert N^{-}_{2}\right\Vert _{\infty},\left\Vert N^{+}_{3}\right\Vert _{\infty},\left\Vert N^{+}_{4}\right\Vert _{\infty}\Biggr\}<R_{0}$
and $\left\Vert N^{\tau,\tau'}\right\Vert \leq R_{\tau,\tau',q}<R_{0}$;
hence $q_{f}<R_{0}.$ 
\end{rem}

\section{Existence theorem for the unbounded time interval}\label{sec:Existence-theorem-for-1}

Let us put 
\begin{equation}
g\left(q,R_{0}\right)\equiv\min\left\{ \frac{1}{8cSR_{0}+\frac{\left(1+c\right)^{2}}{1+2c}16cSq},\frac{R_{0}-q}{8cSR^{2}_{0}}\right\} .\label{jsjusu}
\end{equation}

\begin{proof}[{\large\textbf{Proof of Theorem}}{\large{} \ref{thm:ksiisiiq}}]
 Put 
\begin{equation}
q_{0}\equiv\max_{1\leq i\leq4}\Biggl\{\left\Vert N^{0}_{i}\right\Vert _{\infty},\left\Vert N^{-}_{1}\right\Vert _{\infty},\left\Vert N^{-}_{2}\right\Vert _{\infty},\left\Vert N^{+}_{3}\right\Vert _{\infty},\left\Vert N^{+}_{4}\right\Vert _{\infty}\Biggr\}.\label{uqqqj}
\end{equation}
 Let $R_{0}$ be an arbitrary real number such that 
\begin{equation}
R_{0}>q_{0}.\label{sksiis}
\end{equation}
 Put $S_{0}\left(R_{0}\right)\equiv g\left(q_{0},R_{0}\right).$ Consider
the problem $\Sigma_{0,S_{0}\left(R_{0}\right)}.$ Then from Corollary
\ref{cor:lsoosz}, there exists an unique positive piecewise $C^{1}$
solution $N^{0,S_{0}\left(R_{0}\right)}\in\left(E^{0,S_{0}\left(R_{0}\right)}\right)^{4}$
to the problem $\Sigma_{0,S_{0}\left(R_{0}\right)}$ and $\left\Vert N^{0,S_{0}\left(R_{0}\right)}\right\Vert <R_{0}.$
\\
Put $N^{S_{0}\left(R_{0}\right)}_{i}\left(\cdot,\cdot\right)\equiv N^{0,S_{0}\left(R_{0}\right)}_{i}\left(S_{0}\left(R_{0}\right),\cdot,\cdot\right).$
From Remark \ref{rem:hsyys}, the functions $N^{S_{0}\left(R_{0}\right)}_{i}\left(\cdot,\cdot\right)$
are in the class of the initial data $N^{\tau}_{i}$ as defined in
Subsection \ref{subsec:The-class-of}(Condition 1-b). Put 
\begin{equation}
q_{1}\left(R_{0}\right)\equiv\max_{1\leq i\leq4}\Biggl\{\left\Vert N^{S_{0}\left(R_{0}\right)}_{i}\right\Vert _{\infty},\left\Vert N^{-}_{1}\right\Vert _{\infty},\left\Vert N^{-}_{2}\right\Vert _{\infty},\left\Vert N^{+}_{3}\right\Vert _{\infty},\left\Vert N^{+}_{4}\right\Vert _{\infty}\Biggr\}.\label{kkqql}
\end{equation}
Now choose $S_{1}\left(R_{0}\right)$ such that $S_{1}\left(R_{0}\right)-S_{0}\left(R_{0}\right)=g\left(q_{1}\left(R_{0}\right),R_{0}\right).$
We can consider the problem $\Sigma_{S_{0}\left(R_{0}\right),S_{1}\left(R_{0}\right)}$
where the initial data are $N^{S_{0}\left(R_{0}\right)}_{i}.$ Note
that the compatibility conditions \eqref{eq:kqiiq-1}-\eqref{eq:looqp-1-1}
for $\Sigma_{S_{0}\left(R_{0}\right),S_{1}\left(R_{0}\right)}$ are
satisfied: since $N^{0,S_{0}\left(R_{0}\right)}$ is solution to $\Sigma_{0,S_{0}\left(R_{0}\right)},$
the boundary conditions \eqref{eq:losso-1}-\eqref{eq:ikioi-1-1}
for $\Sigma_{0,S_{0}\left(R_{0}\right)}$ when $t=S_{0}\left(R_{0}\right),$
writes 
\begin{align*}
N^{0,S_{0}\left(R_{0}\right)}_{1}\left(S_{0}\left(R_{0}\right),a_{1},y\right)= & N^{-}_{1}\left(S_{0}\left(R_{0}\right),y\right)\\
N^{0,S_{0}\left(R_{0}\right)}_{2}\left(S_{0}\left(R_{0}\right),x,a_{2}\right)= & N^{-}_{2}\left(S_{0}\left(R_{0}\right),x\right)\\
N^{0,S_{0}\left(R_{0}\right)}_{3}\left(S_{0}\left(R_{0}\right),x,b_{2}\right)= & N^{+}_{3}\left(S_{0}\left(R_{0}\right),x\right)\\
N^{0,S_{0}\left(R_{0}\right)}_{4}\left(S_{0}\left(R_{0}\right),b_{1},y\right)= & N^{+}_{4}\left(S_{0}\left(R_{0}\right),y\right)
\end{align*}
and these are the compatibility conditions for $\Sigma_{S_{0}\left(R_{0}\right),S_{1}\left(R_{0}\right)}.$
From Remark \ref{rem:jsuuaaa} 
\begin{equation}
R_{0}>q_{1}\left(R_{0}\right).\label{hsysys}
\end{equation}
 Then, from Corollary \ref{cor:lsoosz}, there exists an unique positive
piecewise $C^{1}$ solution $N^{S_{0}\left(R_{0}\right),S_{1}\left(R_{0}\right)}\in\left(E^{S_{0}\left(R_{0}\right),S_{1}\left(R_{0}\right)}\right)^{4}$
to the problem $\Sigma_{S_{0}\left(R_{0}\right),S_{1}\left(R_{0}\right)}$
and $\left\Vert N^{S_{0}\left(R_{0}\right),S_{1}\left(R_{0}\right)}\right\Vert <R_{0}.$\\
 Suppose there are $S_{0}\left(R_{0}\right),S_{1}\left(R_{0}\right),\cdots,S_{n}\left(R_{0}\right)>0$
such that $\forall\nu=0,1,\cdots,n-1,$

{*} there exists an unique positive piecewise $C^{1}$ solution $N^{S_{\nu}\left(R_{0}\right),S_{\nu+1}\left(R_{0}\right)}\in\left(E^{S_{\nu}\left(R_{0}\right),S_{\nu+1}\left(R_{0}\right)}\right)^{4}$
to the problem $\Sigma_{S_{\nu}\left(R_{0}\right),S_{\nu+1}\left(R_{0}\right)}$
with the initial data being defined by 
\[
N^{S_{\nu}\left(R_{0}\right)}_{i}\left(\cdot,\cdot\right)\equiv N^{S_{\nu-1}\left(R_{0}\right),S_{\nu}\left(R_{0}\right)}_{i}\left(S_{\nu}\left(R_{0}\right),\cdot,\cdot\right)
\]
 which are consequently in the class of the initial data $N^{\tau}_{i}$
as defined in Subsection \ref{subsec:The-class-of}-Condition 1-b
(Remark \ref{rem:hsyys});

{*} $S_{\nu+1}\left(R_{0}\right)-S_{\nu}\left(R_{0}\right)=g\left(q_{\nu+1}\left(R_{0}\right),R_{0}\right)$
where 
\begin{equation}
q_{\nu+1}\left(R_{0}\right)={\displaystyle \max_{1\leq i\leq4}}\Biggl\{\left\Vert N^{S_{\nu}\left(R_{0}\right)}_{i}\right\Vert _{\infty},\left\Vert N^{-}_{1}\right\Vert _{\infty},\left\Vert N^{-}_{2}\right\Vert _{\infty},\left\Vert N^{+}_{3}\right\Vert _{\infty},\left\Vert N^{+}_{4}\right\Vert _{\infty}\Biggr\}\label{jdudu}
\end{equation}

{*} and that $N^{S_{\nu}\left(R_{0}\right),S_{\nu+1}\left(R_{0}\right)}$
satisfies $\left\Vert N^{S_{\nu}\left(R_{0}\right),S_{\nu+1}\left(R_{0}\right)}\right\Vert <R_{0}$.\\
Put 
\begin{equation}
N^{S_{n}\left(R_{0}\right)}_{i}\left(\cdot,\cdot\right)\equiv N^{S_{n-1}\left(R_{0}\right),S_{n}\left(R_{0}\right)}_{i}\left(S_{n}\left(R_{0}\right),\cdot,\cdot\right).\label{ksiisso}
\end{equation}
 $N^{S_{n}\left(R_{0}\right)}_{i}$ is in the class of the initial
data $N^{\tau}_{i}$ in Subsection \ref{subsec:The-class-of}-Condition
1-b and $N^{S_{n}\left(R_{0}\right)}_{i}$ is positive. Choose 
\[
S_{n+1}\left(R_{0}\right)=S_{n}\left(R_{0}\right)+g\left(q_{n+1}\left(R_{0}\right),R_{0}\right).
\]
 The compatibility conditions \eqref{eq:kqiiq-1}-\eqref{eq:looqp-1-1}
for $\Sigma_{S_{n}\left(R_{0}\right),S_{n+1}\left(R_{0}\right)}$
are satisfied: they are the boundary conditions \eqref{eq:losso-1}-\eqref{eq:ikioi-1-1}
for $\Sigma_{S_{n-1}\left(R_{0}\right),S_{n}\left(R_{0}\right)}$
when $t=S_{n}\left(R_{0}\right).$ Now put 
\begin{equation}
q_{n+1}\left(R_{0}\right)\equiv{\displaystyle \max_{1\leq i\leq4}}\Biggl\{\left\Vert N^{S_{n}\left(R_{0}\right)}_{i}\right\Vert _{\infty},\left\Vert N^{-}_{1}\right\Vert _{\infty},\left\Vert N^{-}_{2}\right\Vert _{\infty},\left\Vert N^{+}_{3}\right\Vert _{\infty},\left\Vert N^{+}_{4}\right\Vert _{\infty}\Biggr\}.\label{kdieei}
\end{equation}
We have $\left\Vert N^{S_{n}\left(R_{0}\right)}_{i}\right\Vert _{\infty}\leq\left\Vert N^{S_{n-1}\left(R_{0}\right),S_{n}\left(R_{0}\right)}_{i}\right\Vert _{\infty}<R_{0}$,
hence $R_{0}>q_{n+1}\left(R_{0}\right)$. From Corollary \ref{cor:lsoosz},
there exists an unique positive piecewise $C^{1}$ solution $N^{S_{n}\left(R_{0}\right),S_{n+1}\left(R_{0}\right)}\in\left(E^{S_{n}\left(R_{0}\right),S_{n+1}\left(R_{0}\right)}\right)^{4}$
to the problem $\Sigma_{S_{n}\left(R_{0}\right),S_{n+1}\left(R_{0}\right)}$
and $\left\Vert N^{S_{n}\left(R_{0}\right),S_{n+1}\left(R_{0}\right)}\right\Vert <R_{0}.$
\\
We thus have a sequence $\left(N^{S_{n}\left(R_{0}\right),S_{n+1}\left(R_{0}\right)}\right)_{n\geq0}$
of positive piecewise $C^{1}$ solutions to the kinetic equations
\eqref{eq:koiqmn} with the boundary conditions \eqref{eq:losso}-\eqref{eq:ikioi-1}
such that for all $n\geq0,$ 
\[
N^{S_{n}\left(R_{0}\right),S_{n+1}\left(R_{0}\right)}\in\left(E^{S_{n}\left(R_{0}\right),S_{n+1}\left(R_{0}\right)}\right)^{4}
\]
 and 
\begin{align*}
N^{S_{n}\left(R_{0}\right),S_{n+1}\left(R_{0}\right)}\left(S_{n+1}\left(R_{0}\right),x,y\right) & =N^{S_{n+1}\left(R_{0}\right),S_{n+2}\left(R_{0}\right)}\left(S_{n+1}\left(R_{0}\right),x,y\right),\\
 & \forall\left(x,y\right)\in\left[a_{1},b_{1}\right]\times\left[a_{2},b_{2}\right].
\end{align*}
From 
\begin{equation}
S_{n+1}\left(R_{0}\right)=S_{n}\left(R_{0}\right)+g\left(q_{n+1}\left(R_{0}\right),R_{0}\right),\label{ksiisi}
\end{equation}
 we deduce that $\left(S_{n}\left(R_{0}\right)\right)_{n\geq0}$ is
a positive and non decreasing real sequence. Hence we have ${\displaystyle \lim_{n\rightarrow\infty}S_{n}\left(R_{0}\right)}=T\left(R_{0}\right)$
where $T\left(R_{0}\right)=+\infty$ or $T\left(R_{0}\right)<+\infty$
. \\
The function $N^{R_{0}}$ whose restriction to 
\[
\left[0,S_{0}\left(R_{0}\right)\right[\times\left[a_{1},b_{1}\right]\times\left[a_{2},b_{2}\right]
\]
 is $N^{0,S_{0}\left(R_{0}\right)}$ and to 
\[
\left[S_{n}\left(R_{0}\right),S_{n+1}\left(R_{0}\right)\right[\times\left[a_{1},b_{1}\right]\times\left[a_{2},b_{2}\right]
\]
 is $N^{S_{n}\left(R_{0}\right),S_{n+1}\left(R_{0}\right)},$$\left(\forall n\geq0\right)$
is defined and continuous on $\left[0,T\left(R_{0}\right)\right[\times\left[a_{1},b_{1}\right]\times\left[a_{2},b_{2}\right]$.
Moreover, 
\[
\left\Vert N^{S_{n}\left(R_{0}\right),S_{n+1}\left(R_{0}\right)}\right\Vert <R_{0},\left(\forall n\geq0\right);
\]
 hence for every $i=1,2,3,4,$$N^{R_{0}}_{i}$ is bounded on $\left[0,S_{0}\left(R_{0}\right)\right[\times\left[a_{1},b_{1}\right]\times\left[a_{2},b_{2}\right]$
and 
\begin{equation}
\left\Vert N^{R_{0}}\right\Vert <R_{0}.\label{jjsususu}
\end{equation}
Let us prove that for every $\left[\tau,\tau'\right]\subset\left[0,T\left(R_{0}\right)\right[,$
the restriction of $N^{R_{0}}$ to $\left[\tau,\tau'\right]\times\left[a_{1},b_{1}\right]\times\left[a_{2},b_{2}\right]$
is in $\left(E^{\tau,\tau'}\right)^{4}.$ \\
Let $\left[\tau,\tau'\right]\subset\left[0,T\left(R_{0}\right)\right[$
be given. Since ${\displaystyle \lim_{n\rightarrow\infty}S_{n}\left(R_{0}\right)}=T\left(R_{0}\right)$,
there exists $n_{0}$ such that 
\[
\forall n\geq n_{0},S_{n}\left(R_{0}\right)\in\left]\tau',T\left(R_{0}\right)\right[.
\]
Hence there exist finitely many $n,n+1\cdots,n+p\in\mathbb{N},p\geq1$
such that 
\[
S_{n}\left(R_{0}\right)\leq\tau\leq\cdots\leq\tau'\leq S_{n+p}\left(R_{0}\right).
\]
By the construction of the function $N^{R_{0}}$, for every $i=1,2,3,4,$
the restriction of $N^{R_{0}}_{i}$ to 
\[
\left[S_{n+k}\left(R_{0}\right),S_{n+k+1}\left(R_{0}\right)\right]\times\left[a_{1},b_{1}\right]\times\left[a_{2},b_{2}\right],k=0,\cdots,p-1
\]
 is in $E^{S_{n+k}\left(R_{0}\right),S_{n+k+1}\left(R_{0}\right)}.$
Therefore the restriction of $N^{R_{0}}_{i}$ is in $E^{\tau,\tau'}$
(Definition \ref{def:nshhsiz}). Hence, the restriction of $N^{R_{0}}$
to $\left[\tau,\tau'\right]\times\left[a_{1},b_{1}\right]\times\left[a_{2},b_{2}\right]$
is in $\left(E^{\tau,\tau'}\right)^{4}.$ \\
Therefore, $N^{R_{0}}$ is a maximal piecewise $C^{1}$ solution to
the problem $\Sigma$ such that 

{*} the domain of $N^{R_{0}}$ is $\left[0,T\left(R_{0}\right)\right[\times\left[a_{1},b_{1}\right]\times\left[a_{2},b_{2}\right]$

{*} $N^{R_{0}}$ is continuous;

{*}for every $\left[\tau,\tau'\right]\subset\left[0,T\left(R_{0}\right)\right[$
the restriction of $M$ to $\left[\tau,\tau'\right]\times\left[a_{1},b_{1}\right]\times\left[a_{2},b_{2}\right]$
is in $\left(E^{\tau,\tau'}\right)^{4}.$ By Theorem \ref{thm:uniqueness}
of uniqueness, $N^{R_{0}}$ is unique; therefore $N^{R_{0}}$ and
$T\left(R_{0}\right)$ are independent of $R_{0}$. Let us put 
\begin{equation}
N^{R_{0}}\equiv N\label{jjsu}
\end{equation}
 and 
\begin{equation}
T\left(R_{0}\right)\equiv T.\label{ksiiz}
\end{equation}
Now ( \eqref{sksiis} ) $R_{0}$ is an arbitrary real number such
that $R_{0}>q_{0}.$ Hence from \eqref{jjsususu}, it follows that
\begin{equation}
\forall R_{0}>q_{0},\;\left\Vert N\right\Vert <R_{0}.\label{juuzu}
\end{equation}
Therefore $\left\Vert N\right\Vert $ is a lower bound of the interval
$\left]q_{0},+\infty\right[;$ hence
\begin{equation}
\left\Vert N\right\Vert \leq q_{0}.\label{usjsju}
\end{equation}
But 
\[
q_{0}=\max_{1\leq i\leq4}\Biggl\{\left\Vert N^{0}_{i}\right\Vert _{\infty},\left\Vert N^{-}_{1}\right\Vert _{\infty},\left\Vert N^{-}_{2}\right\Vert _{\infty},\left\Vert N^{+}_{3}\right\Vert _{\infty},\left\Vert N^{+}_{4}\right\Vert _{\infty}\Biggr\}
\]
and the piecewise $C^{1}$ solution $N$ to $\Sigma$ on $\left[0;T\right[\times\left[a_{1};b_{1}\right]\times\left[a_{2};b_{2}\right]$
satisfies the the following initial and boundary conditions
\begin{align}
N_{i}\left(0,x,y\right)= & N^{0}_{i}\left(x,y\right),\;\left(x,y\right)\in\left[a_{1};b_{1}\right]\times\left[a_{2};b_{2}\right],i=1,\cdots,4\label{eq:lsos-2}\\
N_{1}\left(t,a_{1},y\right)= & N^{-}_{1}\left(t,y\right),\;\left(t,y\right)\in\left[0;T\right[\times\left[a_{2};b_{2}\right]\label{eq:losso-2}\\
N_{2}\left(t,x,a_{2}\right)= & N^{-}_{2}\left(t,x\right),\;\left(t,x\right)\in\left[0;T\right[\times\left[a_{1};b_{1}\right]\label{eq:lsoo-2}\\
N_{3}\left(t,x,b_{2}\right)= & N^{+}_{3}\left(t,x\right),\;\left(t,x\right)\in\left[0;T\right[\times\left[a_{1};b_{1}\right]\label{eq:ikioi-3}\\
N_{4}\left(t,b_{1},y\right)= & N^{+}_{4}\left(t,y\right),\;\left(t,y\right)\in\left[0;T\right[\times\left[a_{2};b_{2}\right],\label{eq:ikioi-1-2}
\end{align}
hence 
\begin{equation}
\left\Vert N^{0}_{i}\right\Vert _{\infty},\left\Vert N^{-}_{1}\right\Vert _{\infty},\left\Vert N^{-}_{2}\right\Vert _{\infty},\left\Vert N^{+}_{3}\right\Vert _{\infty},\left\Vert N^{+}_{4}\right\Vert _{\infty}\leq\left\Vert N\right\Vert ;\label{kksisi}
\end{equation}
therefore
\begin{equation}
q_{0}\leq\left\Vert N\right\Vert .\label{ujssu}
\end{equation}
We deduce that 
\begin{equation}
\left\Vert N\right\Vert =q_{0}.\label{usjsju-1}
\end{equation}
Let us prove that $T=+\infty.$\\
Suppose by contradiction that $T<+\infty.$ Let us take a $R_{0}>q_{0}.$
We have $T=T\left(R_{0}\right)={\displaystyle \lim_{n\rightarrow\infty}S_{n}\left(R_{0}\right)}.$
\\
We have 
\begin{align}
S_{n+1}\left(R_{0}\right) & =S_{n}\left(R_{0}\right)+g\left(q_{n+1}\left(R_{0}\right),R_{0}\right)\label{lsoso-1}\\
S_{n}\left(R_{0}\right) & =S_{n-1}\left(R_{0}\right)+g\left(q_{n}\left(R_{0}\right),R_{0}\right)\nonumber \\
....... & .........\nonumber \\
S_{1}\left(R_{0}\right) & =S_{0}\left(R_{0}\right)+g\left(q_{1}\left(R_{0}\right),R_{0}\right)\nonumber \\
S_{0}\left(R_{0}\right) & =0+g\left(q_{0}\left(R_{0}\right),R_{0}\right)
\end{align}
hence 
\begin{equation}
S_{n}\left(R_{0}\right)=\sum^{n}_{k=0}g\left(q_{k}\left(R_{0}\right),R_{0}\right).\label{ksius}
\end{equation}
$\lim_{n\rightarrow\infty}S_{n}\left(R_{0}\right)<\infty$ implies
that $\lim_{n\rightarrow\infty}g\left(q_{n}\left(R_{0}\right),R_{0}\right)=0$.
\\
From \eqref{jsjusu}
\[
g\left(q_{n}\left(R_{0}\right),R_{0}\right)=\min\left\{ \frac{1}{8cSR_{0}+\frac{\left(1+c\right)^{2}}{1+2c}16cSq_{n}\left(R_{0}\right)},\frac{R_{0}-q_{n}\left(R_{0}\right)}{8cSR^{2}_{0}}\right\} .
\]
From \eqref{ksiisso} and \eqref{kdieei}, we have 
\begin{equation}
q_{n}\left(R_{0}\right)={\displaystyle \max_{1\leq i\leq4}}\Biggl\{\left\Vert N^{S_{n-1}\left(R_{0}\right),S_{n}\left(R_{0}\right)}_{i}\left(S_{n}\left(R_{0}\right),\cdot,\cdot\right)\right\Vert _{\infty},\left\Vert N^{-}_{1}\right\Vert _{\infty},\left\Vert N^{-}_{2}\right\Vert _{\infty},\left\Vert N^{+}_{3}\right\Vert _{\infty},\left\Vert N^{+}_{4}\right\Vert _{\infty}\Biggr\}.\label{kdieei-1}
\end{equation}
But 
\[
\left\Vert N^{S_{n-1}\left(R_{0}\right),S_{n}\left(R_{0}\right)}_{i}\left(S_{n}\left(R_{0}\right),\cdot,\cdot\right)\right\Vert _{\infty},\left\Vert N^{-}_{1}\right\Vert _{\infty},\left\Vert N^{-}_{2}\right\Vert _{\infty},\left\Vert N^{+}_{3}\right\Vert _{\infty},\left\Vert N^{+}_{4}\right\Vert _{\infty}\leq\left\Vert N\right\Vert =q_{0}
\]
 hence 
\[
q_{n}\left(R_{0}\right)\leq q_{0},\forall n\geq1.
\]
 Hence 
\[
\frac{1}{8cSR_{0}+\frac{\left(1+c\right)^{2}}{1+2c}16cSq_{n}\left(R_{0}\right)}\geq\frac{1}{8cSR_{0}+\frac{\left(1+c\right)^{2}}{1+2c}16cSq_{0}}
\]
and 
\[
\frac{R_{0}-q_{n}\left(R_{0}\right)}{8cSR^{2}_{0}}\geq\frac{R_{0}-q_{0}}{8cSR^{2}_{0}}.
\]
Therefore 
\[
g\left(q_{n}\left(R_{0}\right),R_{0}\right)\geq\min\left\{ \frac{1}{8cSR_{0}+\frac{\left(1+c\right)^{2}}{1+2c}16cSq_{0}},\frac{R_{0}-q_{0}}{8cSR^{2}_{0}}\right\} 
\]
hence if $\left(g\left(q_{n}\left(R_{0}\right),R_{0}\right)\right)_{n\geq0}$
converges then 
\[
\lim_{n\rightarrow\infty}g\left(q_{n}\left(R_{0}\right),R_{0}\right)\geq\min\left\{ \frac{1}{8cSR_{0}+\frac{\left(1+c\right)^{2}}{1+2c}16cSq_{0}},\frac{R_{0}-q_{0}}{8cSR^{2}_{0}}\right\} >0.
\]
 We conclude that 
\[
T=\lim_{n\rightarrow\infty}S_{n}\left(R_{0}\right)=+\infty.
\]
 
\end{proof}

\section{Conclusion}

In this paper, we investigated the existence and uniqueness of piecewise
classical solutions to initial-boundary value problems associated
with discrete kinetic equations, in particular those arising from
the four-velocity Broadwell model. The data are in class of positive
piecewise $C^{1}$ bounded functions with bounded first-order partial
derivatives. Our approach was based on fixed-point theorems, carefully
selected according to the structural properties of the operator obtained
in the reformulation of the problem as fixed-point equation. We then
proved that the initial-boundary problems admits an unique global
positive bounded piecewise $C^{1}$ solution for arbitrary large value
of the initial and boundary data.

Our results show that fixed-point methods provide an effective framework
for addressing this class of nonlinear problems in discrete kinetic
theory. 

A natural continuation of this work would be the extension of the
developed method to the general multidimensional discrete kinetic
equations, to more general domains, and to more general mixed boundary
conditions. Another direction of investigation concerns the construction
of exact solutions and the numerical resolution of the discrete kinetic
equations, supported by the uniqueness results established in this
paper.

\end{document}